\documentclass[10pt]{amsart}

\pdfoutput=1

\allowdisplaybreaks

\usepackage{amsfonts}
\usepackage{amsmath}
\usepackage{amssymb}
\usepackage{amsthm}
\usepackage{amsbsy}
\usepackage{xcolor}

\usepackage{graphicx} \usepackage{enumerate} \usepackage{multicol}
\usepackage{mathrsfs} \usepackage[all,cmtip]{xy}
\usepackage[color=blue!20,textsize=footnotesize]{todonotes}
\usepackage{hyperref}

\usepackage{algorithm, algpseudocode}

\newcounter{dummy} \numberwithin{dummy}{section}
\newtheorem{theorem}[dummy]{Theorem}
\newtheorem{corollary}[dummy]{Corollary}
\newtheorem{lemma}[dummy]{Lemma}

\newtheorem{proposition}[dummy]{Proposition}
\theoremstyle{remark}
\newtheorem{remark}[dummy]{Remark}

\newcommand{\calX}{\mathcal{X}}

\newcommand{\calA}{\mathcal{A}}

\newcommand{\scrE}{\mathscr{E}}

\newcommand{\bfS}{\mathbf{S}}
\newcommand{\bfsigma}{\boldsymbol \sigma}

\DeclareMathOperator{\Ann}{Ann}

\DeclareMathOperator{\sgn}{sgn}

\DeclareMathOperator{\rank}{rank}
\DeclareMathOperator{\spn}{span}

\newcommand{\E}{\mathbb{E}}
\newcommand{\pp}{\mathbb{P}}

\DeclareMathOperator{\length}{length}
\DeclareMathOperator{\dv}{div}

\newcommand{\R}{\mathbb{R}}
\newcommand{\N}{\mathbb{N}}
\newcommand{\qq}{\mathbb{Q}}
\newcommand{\dd}{\,{\mathrm d}}
\newcommand{\db}{{\mathrm d}}

\newcommand{\e}{\operatorname{e}}
\newcommand{\im}{\mathrm{i}}
\newcommand{\soo}{\tau_0}

\numberwithin{equation}{section}

\title{Score matching for sub-Riemannian bridge sampling}
\author[E.~Grong, K.~Habermann, S.~Sommer]{Erlend~Grong, Karen~Habermann, Stefan~Sommer}

\address{Erlend~Grong, University of Bergen, Department of Mathematics, P.O. Box 7803, 5020 Bergen, Norway.}
\email{erlend.grong@uib.no}

\address{Karen Habermann, Department of Statistics, University of Warwick, Coventry, CV4 7AL, United Kingdom.}
\email{karen.habermann@warwick.ac.uk}

\address{Stefan Sommer, Department of Computer Science, University of Copenhagen, Universitetsparken 5, DK-2100 Copenhagen E, Denmark.}
\email{sommer@di.ku.dk}

\date{}

\thanks{The first author is supported by the grant GeoProCo from the Trond Mohn Foundation - Grant TMS2021STG02 (GeoProCo). The third author is supported by a research grant (VIL40582) from VILLUM FONDEN and the Novo Nordisk Foundation grant NNF18OC0052000.}

\hypersetup{
  pdftitle={Score matching for sub-Riemannian bridge sampling},
  pdfauthor={Erlend~Grong, Karen~Habermann, Stefan~Sommer},
}

\begin{document}

\begin{abstract}
    Simulation of conditioned diffusion processes is an essential tool in inference for stochastic processes, data imputation, generative modelling, and geometric statistics. Whilst simulating diffusion bridge processes is already difficult on Euclidean spaces, when considering diffusion processes on Riemannian manifolds the geometry brings in further complications. In even higher generality, advancing from Riemannian to sub-Riemannian geometries introduces hypoellipticity, and the possibility of finding appropriate explicit approximations for the score of the diffusion process is removed. We handle these challenges and construct a method for bridge simulation on sub-Riemannian manifolds by demonstrating how recent progress in machine learning can be modified to allow for training of score approximators on sub-Riemannian manifolds. Since gradients dependent on the horizontal distribution, we generalise the usual notion of denoising loss to work with non-holonomic frames using a stochastic Taylor expansion, and we demonstrate the resulting scheme both explicitly on the Heisenberg group and more generally using adapted coordinates. We perform numerical experiments exemplifying samples from the bridge process on the Heisenberg group and the concentration of this process for small time.
\end{abstract}

\maketitle

\section{Introduction}

The problem of simulating bridges for diffusion processes has been studied extensively in the literature, such as in~\cite{BeskosPapaspiliopoulosRobertsFearnhead,delyonhu,bridgesim3,bladt2016,bridgesim6,bridgesim5}, and effective simulation schemes for diffusion bridge processes have become paramount for data imputation and in parameter inference for stochastic processes as they provide a stochastic method for approximating the intractable likelihood or when sampling from a posteriori distribution of parameters.
In the parameter estimation from given sets of time-discrete observations, bridge simulations feature in the data imputation performed when missing data is simulated through diffusion bridge processes that connect the observed data, see e.g. Golightly and Wilkinson~\cite{bayes0}, Papaspiliopoulos, Roberts and Stramer~\cite{bayes1}, and van der Meulen and Schauer~\cite{bayes2}.

Bridge sampling further plays a prominent role in geometric statistics and has found applications in areas such as medical image analysis and shape analysis, see~\cite{shape1}. As discussed in~\cite{medical2}, simulated diffusion bridge processes have been used to approximate the heat kernel of a diffusion process on a Riemannian manifold and to estimate the underlying metric structure.

Most previous work on diffusion bridge simulations focuses on elliptic diffusion processes on Euclidean spaces where the diffusivity matrix is uniformly invertible. The simulations introduced by Bierkens, van der Meulen and Schauer~\cite{bridgesim5} apply both to elliptic diffusion processes and to those hypoelliptic diffusion processes whose hypoellipticity arises from an interaction of the drift term with the diffusivity, such as for Langevin dynamics.
Although the recent papers \cite{JS21,buiInferencePartiallyObserved2023,corstanjeSimulatingConditionedDiffusions2024} have addressed bridge simulations on manifolds, these schemes do not apply directly to simulate bridge processes associated with diffusion processes on sub-Riemannian manifolds, that is, those hypoelliptic diffusion processes where the hypoellipticity is induced by the diffusivity term itself, without any need to interact with the drift term. Phrased differently, as far as the authors are aware, it has not been known how to effectively simulate bridge processes on sub-Riemannian manifolds. Such spaces are described as a triple $(M,E,g)$, where $M$ is a smooth manifold, $E\subseteq TM$ is a bracket-generting subbundle that is called horizontal bundle and $g$ is a smoothly varying inner product on $E$. The subbundle $E$ represents the directions in which we allow the stochastic process on $M$ to diffuse and the bracket-generating assumption on $E$ ensures that the induced diffusion process is hypoelliptic. Sub-Riemannian structures naturally appear in all sciences, e.g. in the study of constrained physical systems such as the motion of robot arms or the orbital dynamics of satellites.

Motivated by the apparent gap in existing bridge simulation schemes, the immersiveness of sub-Riemannian geometry, and the importance of bridge simulation for data imputation and parameter inference, we here develop a method for bridge simulation on sub-Riemannian manifolds.

\smallskip

A central object in the study of diffusion bridge processes is \emph{the score} of the density of the corresponding unconditioned diffusion process. To introduce the score associated with a diffusion process on a manifold $M$, we need a gradient $\nabla^E$ on $M$. On a Riemannian manifold, $\nabla^E$ is simply taken to be the usual gradient, whilst on a sub-Riemannian manifold $(M,E,g)$, the horizontal gradient $\nabla^E$ will be defined with respect to the subbundle $E$.
For a time-homogeneous diffusion process on $M$ starting from $x_0\in M$ and having transition probability density $p_t\colon M\times M\to\R$ for $t>0$, the score $S_t(x_0,\cdot)\colon M\to E$ is given by, for $y\in M$,
\begin{equation} \label{eq:score}
S_t(x_0,y) = \nabla^{y,E} \log p_t(x_0, y).
\end{equation}
Notably, the score is generally intractable due to its dependency on the transition probability densities.

Bridge simulation schemes often employ guiding terms that approximate the intractable score of the diffusion process. As we do not have such approximations available explicitly in the sub-Riemannian context, we employ the approach of \cite{HdBDT21} building on score matching techniques \cite{hyvarinenEstimationNonNormalizedStatistical2005,vincentConnectionScoreMatching2011} to learn the score parameterised by a neural network and use this score in the simulation of the conditioned diffusion processes. The guiding terms used in previous works on Riemannian manifolds include radial vector fields, see~\cite{JS21,buiInferencePartiallyObserved2023}, or exploit the heat kernel on comparison manifolds as in~\cite{corstanjeSimulatingConditionedDiffusions2024}. Neither of these options are computationally tractable in the sub-Riemannian context. For training the neural network score approximation, we show how the score divergence as well as a denoising term can be used for the loss function. Both approaches are used in the Euclidean setting, but we will see that the sub-Riemannian geometry influences the terms resulting in non-trivial differences between the Euclidean and geometric losses, see Figure~\ref{fig:illustration}.

\begin{figure}[htp] 
    \centering
    \resizebox{.5\columnwidth}{!}
{

\tikzset{every picture/.style={line width=0.75pt}} 

\begin{tikzpicture}[x=0.75pt,y=0.75pt,yscale=-1,xscale=1]

\draw    (147.38,214.88) -- (320.72,159.73) ;
\draw [shift={(322.63,159.13)}, rotate = 162.35] [color={rgb, 255:red, 0; green, 0; blue, 0 }  ][line width=0.75]    (10.93,-3.29) .. controls (6.95,-1.4) and (3.31,-0.3) .. (0,0) .. controls (3.31,0.3) and (6.95,1.4) .. (10.93,3.29)   ;
\draw    (255.25,260.75) -- (206.79,86.18) ;
\draw [shift={(206.25,84.25)}, rotate = 74.48] [color={rgb, 255:red, 0; green, 0; blue, 0 }  ][line width=0.75]    (10.93,-3.29) .. controls (6.95,-1.4) and (3.31,-0.3) .. (0,0) .. controls (3.31,0.3) and (6.95,1.4) .. (10.93,3.29)   ;
\draw [line width=1.5]    (183,218) -- (202,199) ;
\draw [line width=1.5]    (202,199) -- (221,205.79) ;
\draw [line width=1.5]    (221,205.79) -- (217,192) ;
\draw [line width=1.5]    (217,192) -- (226,204) ;
\draw [line width=1.5]    (226,204) -- (235,187) ;
\draw [line width=1.5]    (235,187) -- (254,182) ;
\draw [line width=1.5]    (254,182) -- (267,192) ;
\draw [line width=1.5]    (267,192) -- (289,193) ;
\draw  [fill={rgb, 255:red, 155; green, 155; blue, 155 }  ,fill opacity=0.38 ] (227.38,137.63) .. controls (247.38,177.63) and (248.38,160.13) .. (254.13,183.88) .. controls (259.88,207.63) and (249.54,197.13) .. (256.88,249.13) .. controls (264.21,301.13) and (269.46,316.54) .. (235.13,188.88) .. controls (200.79,61.21) and (207.38,97.63) .. (227.38,137.63) -- cycle ;
\draw  [fill={rgb, 255:red, 128; green, 128; blue, 128 }  ,fill opacity=0.54 ] (209.75,175.75) .. controls (216.67,163.67) and (237.25,155.75) .. (254.25,163.25) .. controls (271.25,170.75) and (277,172.33) .. (304.67,165) .. controls (332.33,157.67) and (178.42,205.58) .. (178.33,205.67) .. controls (178.25,205.75) and (202.83,187.83) .. (209.75,175.75) -- cycle ;
\draw [line width=1.5]    (184.25,233.75) -- (183,218) ;
\draw [line width=1.5]    (170.25,241.75) -- (184.25,233.75) ;
\draw [line width=1.5]    (289,193) -- (295.75,202.25) ;
\draw [line width=1.5]    (295.75,202.25) -- (312.75,192.25) ;
\draw  [dash pattern={on 4.5pt off 4.5pt}]  (190.38,126.88) -- (322.73,239.45) ;
\draw [shift={(324.25,240.75)}, rotate = 220.38] [color={rgb, 255:red, 0; green, 0; blue, 0 }  ][line width=0.75]    (10.93,-3.29) .. controls (6.95,-1.4) and (3.31,-0.3) .. (0,0) .. controls (3.31,0.3) and (6.95,1.4) .. (10.93,3.29)   ;
\draw  [dash pattern={on 4.5pt off 4.5pt}]  (194.75,251.75) -- (317.47,103.79) ;
\draw [shift={(318.75,102.25)}, rotate = 129.67] [color={rgb, 255:red, 0; green, 0; blue, 0 }  ][line width=0.75]    (10.93,-3.29) .. controls (6.95,-1.4) and (3.31,-0.3) .. (0,0) .. controls (3.31,0.3) and (6.95,1.4) .. (10.93,3.29)   ;

\draw (329,161) node [anchor=north west][inner sep=0.75pt]   [align=left] {$\displaystyle E_{X_{t_{i}}}$};
\draw (173,74) node [anchor=north west][inner sep=0.75pt]   [align=left] {$\displaystyle V_{X}{}_{_{t_{i}}}$};
\draw (311.5,245) node [anchor=north west][inner sep=0.75pt]   [align=left] {$\displaystyle E_{X_{t_{i+1}}}$};
\draw (310,77.5) node [anchor=north west][inner sep=0.75pt]   [align=left] {$\displaystyle V_{X}{}_{_{t_{i+1}}}$};
\draw (155,240.75) node [anchor=north west][inner sep=0.75pt]   [align=left] {$\displaystyle X_{t}$};

\end{tikzpicture}

}

  \caption{In a sub-Riemannian geometry, the increments of a horizontal stochastic process $(X_t)_{t\geq 0}$ lie in the horizontal bundle $E$. When approximating the score for short time intervals by taking derivatives of the density of the steps, a common approximation takes the increment at the $i$th step to be normally distributed in the distribution $E_{X_{t_i}}$ at $X_{t_i}$ (sketched dark grey). However, this distribution is not differentiable with horizontal derivatives in the distribution $E_{X_{t_{i+1}}}$ at step $X_{t_{i+1}}$. The hypoellipticity of $(X_t)_{t\geq 0}$ implies that the distribution has a vertical component in $V_{X_{t_i}}$ (sketched light grey). We rely on Taylor expansion for the stochastic integral to approximate this component, which we then exploit to derive a sub-Riemannian denoising loss for use in the training of neural network score approximators.}
  \label{fig:illustration}
\end{figure}
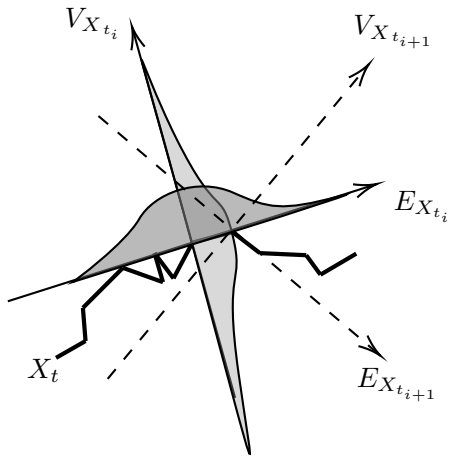

\subsection{Related work}
In Euclidean space, simulation of conditioned diffusion processes has been treated extensively in the literature in works including \cite{durham2002,Stuart,BeskosPapaspiliopoulosRobertsFearnhead,delyonhu,beskos-mcmc-methods,GolightlyWilkinsonChapter,LinChenMykland,bridgesim3,bladt2016, whitaker2017,bridgesim6,bridgesim5, bierkens2020piecewise, mider2021continuous}.
The use of explicit score approximations has extended bridge simulations to Riemannian manifolds, see \cite{jensen2019simulation,JS21,jensen2022discrete,buiInferencePartiallyObserved2023,corstanjeSimulatingConditionedDiffusions2024}. Instead of using explicit score approximations, the work \cite{HdBDT21} has demonstrated how score matching in \cite{hyvarinenEstimationNonNormalizedStatistical2005,vincentConnectionScoreMatching2011} can be exploited to train a neural network to approximate the score and use this in the simulation. Score matching and generative modelling on manifolds has been the focus of \cite{huangRiemannianDiffusionModels2022,bortoliRiemannianScoreBasedGenerative2022}.
In this paper, we build on elements of these ideas to show how a neural network can learn an approximated score on sub-Riemannian manifolds, which subsequently can be used for bridge sampling.

\subsection{Paper outline}
In Section~\ref{sec:geometry_diffusion}, we discuss the needed background material for the paper, and, in Section~\ref{sec:bridge}, we provide further details on bridge processes and their time reversals in the sub-Riemannian setting linking their generators and the score. We apply this in Section~\ref{sec:score} to construct the neural network score approximator using divergence loss, followed by a derivation of the corresponding denoising loss in the sub-Riemannian context in Section~\ref{sec:denoising}. 
Combined together, the stochastic differential equation for the conditioned process and the score approximation allow for numerical simulations of the bridge processes, which we use in Section~\ref{sec:experiments} to exemplify score simulations and bridge sampling on the Heisenberg group.

\section{Diffusion processes on sub-Riemannian manifolds}
\label{sec:geometry_diffusion}
We here discuss relevant background material on sub-Riemannian geometry and diffusion processes on sub-Riemannian manifolds needed in the remainder of the paper.
\subsection{Sub-Riemannian manifolds}
Let $M$ be a connected manifold. We consider a fixed subbundle $E \subseteq TM$ that we denote \emph{the horizontal bundle}.
We refer to the elements of $E$ as \emph{horizontal vectors} and we use $\calX(E)$ to denote the space of vector fields that take values in $E$. Let $g = \langle \cdot , \cdot \rangle_g$ be a smoothly varying inner product defined only on $E$. The triple $(M, E,g)$ is called \emph{a sub-Riemannian manifold}, and the smoothly varying tensor $g$ is called \emph{a sub-Riemannian metric}. Such an inner product gives rise to \emph{a sharp map}
$$\sharp\colon T^*M \to E \qquad \text{defined by} \qquad \alpha(v) = \langle \sharp \alpha, v\rangle_g, \qquad \alpha \in T^*M, v \in E.$$
Note that $\sharp$ is not invertible when $E \neq TM$. A sub-Riemannian metric $g$ can also be described in terms of its cometric $g^*$.
This cometric $g^*=\langle \cdot , \cdot \rangle_{g^*}$ is defined for any pair of covectors $\alpha$, $\beta \in T_x^*M$ for $x \in M$ by
$$\langle \alpha, \beta \rangle_{g^*} = \langle \sharp \alpha, \sharp \beta \rangle_{g}.$$
Whilst the sub-Riemannian cometric $g^*$ is defined on all of $T^*M$, it is degenerate whenever $E \neq TM$ because it vanishes on $\ker \sharp = \Ann(E)$, which consists of all covectors vanishing on $E$.

We say that an absolutely continuous curve $\gamma\colon[a,b]\to M$ is \emph{horizontal} if it is the case that $\dot \gamma(t) \in E_{\gamma(t)}$ for almost every $t \in [a,b]$. Such a curve has a well-defined length given by
$$\length(\gamma) = \int_a^b \| \dot \gamma(t)  \|_g \dd t :=\int_a^b \langle \dot \gamma(t) , \dot \gamma(t) \rangle_g^{1/2} \dd t.$$
Furthermore, the metric defines a distance $d_g\colon M\times M\to\R$ by setting $d_g(x,y)$ for $x,y\in M$ to be the infimum over the lengths of all those horizontal curves which connect $x$ to $y$. In order to guarantee the property that $d_g(x,y)$ is finite for all $x,y\in M$, we assume for the remainder of the paper that $E$ is \emph{bracket-generating}, that is, we suppose that all directional derivatives can be recreated from $E$. To be more precise, let us define, for any $x\in M$,
$$\bar{E}_x = \spn \left\{ [\sigma_1,[\sigma_2,\dots, [\sigma_{l-1},\sigma_{l}]\cdots ]] (x) : \sigma_j \in \calX(E),\enspace j\in\{1,\dots, l\},\enspace l\geq 1 \right\},$$
where we interpret the bracket for the case $l=1$ simply as $\sigma_1$. The vector bundle $E$ is then said to be bracket-generating if $\bar{E} = TM$. The Chow--Rashevski{\u\i} theorem in \cite{Cho39,Ras38} tells us that the bracket-generating assumption is a sufficient condition not only for $d_g$ to be finite, but furthermore for its metric topology to have the same open sets as the original topology of the manifold $M$.

Observe that we can consider a Riemannian manifold $(M,g)$ as a special case of a sub-Riemannian manifold $(M,E,g)$ with $E = TM$.

\subsection{Sub-Laplacians and diffusion processes}
We consider a sub-Riemannian manifold $(M, E, g)$.
For a function $f \in C^\infty(M)$, define $\nabla^E f = \sharp \dd f \in \calX(E)$, \emph{the horizontal gradient} of $f$. In the special case where $E =TM$, we simply write $\nabla^E f= \nabla f$. Assuming that we also have a smooth volume density $\db\mu$ on $M$, we define a corresponding second order operator $\Delta = \Delta_{E,g,\db\mu}$ by
$$\Delta f = \dv_{\db\mu}(\nabla^E f), \qquad f\in C^\infty(M).$$
The operator $\Delta$ is \emph{the sub-Laplacian} of $(M,E,g,\db\mu)$. It is the unique second order operator on $M$ satisfying, for any pair of compactly supported functions $f_1, f_2 \in C^\infty_0(M)$,
\begin{align*}
    \langle \nabla^E f_1, \nabla^E f_2 \rangle_g & = \frac{1}{2} \left( \Delta (f_1 f_2) - f_1 \Delta f_2 - f_2 \Delta f_1\right),\\
    \int_M \langle \nabla^E f_1, \nabla^E f_2 \rangle_g \dd\mu & = - \int_M f_1 \Delta f_2 \dd\mu = - \int_M f_2 \Delta f_1 \dd\mu.
\end{align*}
If $\rank E=k$ and $(\sigma_1, \dots, \sigma_k)$ is a local orthonormal frame for $E$ with respect to the sub-Riemannian metric $g$, then
$$\Delta = \sum_{j=1}^k (\sigma_j^2 + (\dv_{\db\mu} \sigma_j) \sigma_j).$$
If the subbundle $E$ is a proper subbundle of the tangent bundle $TM$, then the operator $\Delta$ is not elliptic. However, if $E$ is bracket-generating, then both the operator~$\Delta$ and its heat operator $\partial_t - \frac{1}{2} \Delta$ are hypoelliptic, see H\"ormander~\cite{Hor67}.

For the special case where $E  = TM$ and $(M,g)$ is a Riemannian manifold, the operator~$\Delta$ is called the Witten--Laplacian. If $\dd\mu_g$ is the Riemannian volume form of the Riemanian metric $g$ and $\Delta_g$ is the Laplace--Beltrami operator corresponding to $\dd\mu_g$, then there exists a function $\phi$ on $M$ such that $\db\mu = \e^\phi \db\mu_g$, which is related to $\Delta$ by
$$\Delta = \Delta_g + \nabla \phi,$$
where $\nabla \phi$ is defined with respect to the metric $g$.

For the remainder of the paper, we consider an operator $L$ of the form
\begin{displaymath}
    L = \Delta + 2Z,
\end{displaymath}
where $\Delta$ is the sub-Laplacian of $(M,E,g,\db\mu)$ and $Z$ is a smooth vector field on $M$. For such a second order partial differential operator $L$ and $x_0\in M$, there exists a unique semi-martingale $(X_t)_{\zeta>t\geq 0} = (X_t^{x_0})_{\zeta>t\geq 0}$ on $M$ starting from $x_0$, with infinitesimal generator $\frac{1}{2}L$ and defined up to some maximal explosion time $\zeta$, see e.g. Elworthy~\cite{elworthy}, \'Emery~\cite{emery} and Hsu~\cite{hsu}. Throughout the paper, we will assume that the process does in fact not explode, that is, $\zeta(x_0)=\infty$ almost surely. For details on some criteria for non-explosion, see e.g. Bichteler and Jacod~\cite{bichtelerjacod} or Norris~\cite{norris86} as well as Grigor’yan~\cite{grigor2013stochastic} and \cite{grong2019stochastic} for geometric conditions.

\begin{remark}[Choice of measure] \label{re:ChoiceMeasure}
Consider $L = \Delta_{\db\mu}+2Z$ for $\Delta_{\db\mu} = \Delta_{E,g,\db\mu}$ defined with respect to a sub-Riemannian structure $(E,g)$ and a smooth volume density $\db\mu$ on a manifold $M$. We observe that if $\db\nu = \e^{\phi} \db\mu$ is another smooth measure on $M$ then $\Delta_{\db\nu} = \Delta_{E,g,\db\nu}$ satisfies
\begin{displaymath}
    \Delta_{\db\nu}=\Delta_{\db\mu} + \nabla^E \phi,
\end{displaymath}
and hence, we can write
\begin{displaymath}
    L = \Delta_{\db\nu} - \nabla^E \phi + 2Z =: \Delta_{\db\nu} + 2\tilde Z.
\end{displaymath}
It follows that we are free to change our definition of $\db\mu$ as long as we update the vector field $Z$ accordingly.

Conversely, if we have $L = \Delta + 2Z$ where $2Z = \nabla^E \phi$ is a horizontal gradient then $L = \Delta_{E,g,\db\tilde \mu}$ is the sub-Laplacian with respect to some measure $\db\tilde \mu$ on $M$.
\end{remark}

\subsection{Local description of the operators}
The following discussion concerning the local description of Brownian motion on a Riemannian manifold provides a motivation for studying the more general case of sub-Riemannian manifolds. Let $(W_t)_{t\geq 0}$ be a standard Brownian motion on $\mathbb{R}^d$ and consider the unique strong solution $(X_t)_{t\geq 0}$, assumed to exist for all times, to the Stratonovich stochastic differential equation
\begin{equation*}
\db X_t = \sigma(X_t) \circ \db W_t + \sigma_0(X_t) \dd t, \end{equation*}
where $\sigma_0 \colon \mathbb{R}^d \to \mathbb{R}^d$ is a smooth vector field and $\sigma\colon\mathbb{R}^d \to \mathbb{R}^{d\times d}$, $x \mapsto \sigma(x) = (\sigma^i_{j}(x))$ a smooth function. If $\sigma \sigma^\top(x) = (\sum_{l=1}^d \sigma^i_{l}\sigma^j_{l}(x))$ is positive definite for all $x\in\R^d$, then we can define a Riemannian metric $g$ on $\R^d$ by setting $g = (g_{ij}) = (\sigma \sigma^\top )^{-1}$. The infinitesimal generator $\frac{1}{2} L$ of the stochastic process $(X_t)_{t\geq 0}$ is then in terms of the Laplace--Beltrami operator $\Delta_g$ given by
$$L = \sum_{i,j,l=1}^d \sigma^i_{l} \sigma^j_{l}\partial_{x^i} \partial_{x^j} + \sum_{l=1}^d \left(2\sigma_0^l + \sum_{i,j=1}^d  \sigma^i_{j} (\partial_{x^i} \sigma^l_{j}) \right)\partial_{x^l} =: \Delta_g + 2Z,$$
where $\Delta_g = \sum_{i,j,l=1}^d \sigma_l^i \sigma_l^i \partial_{x^i} \partial_{x^j} -\sum_{i,j,l=1}^d g^{ij} \Gamma_{ij}^l \partial_{x^l}$,
\begin{displaymath}
    2Z = \sum_{l=1}^d 2Z^l \partial_{x^l} = \sum_{l=1}^d \left( 2 \sigma^l_0 + \sum_{i,j=1}^d \left(   \sigma^i_{j} (\partial_{x^i} \sigma^l_{j}) +g^{ij} \Gamma_{ij}^l   \right)    \right) \partial_{x^l}
\end{displaymath}
and $\Gamma_{ij}^l$ are the Christoffel symbols with respect to the Riemannian metric $g$ on $\R^d$.
 
If $\sigma \sigma^\top$ is merely positive semi-definite of constant rank, then we cannot invert it. However, we can still use $g^* = \sigma \sigma^\top$ as a cometric for a sub-Riemannian structure $(E,g)$. If we consider $g^*$ as a matrix, which is the same as identifying it with $\sharp$, then $E$ will be the image of $g^*$.

\subsection{Densities and conditional probability}
Let $\Delta$ be the sub-Laplacian of the sub-Riemannian manifold $(M, E,g)$ equipped with the smooth volume density $\db\mu$ and let $Z$ be a smooth vector field on $M$. Define $(X_t)_{t\geq 0}$ to be the unique stochastic process starting from $x_0\in M$ and having generator $\frac{1}{2} L = \frac{1}{2} \Delta + Z$. Recall that we assume this stochastic process does not explode.
For $T>0$ given, let $\pp=\pp^{x_0,T}$ be the measure on the path space $C([0,T], M)$ induced by $(X_{t})_{t\in [0,T]}$. For $s,t\in[0,T]$ with $s<t$, we further let $\pp_{s,t}$ denote the push-forward measure on $M \times M$ of the map $(x_{r})_{r \in [0,T]} \mapsto (x_s, x_t)$. Similarly, we set $\pp_t$ to be the push-forward measure of the evaluation $(x_{r})_{r \in [0,T]} \mapsto x_t$.
Since $E$ is assumed to be bracket-generating, it follows by~\cite{Hor67,StVa72} that $\pp_{s,t}$ and $\pp_t$ have smooth densities with respect to $\db\mu \otimes \db \mu$ and $\db \mu$, respectively.
For any $x,y\in M$ and for any $s,t\in[0,T]$ with $s<t$, we write
$$\db\pp_{s,t}(x,y) = p_s(x) p_{s,t}(x,y) (\db\mu \otimes \db\mu)(x,y)
\quad\text{and}\quad \db\pp_t(x) = p_t(x) \dd\mu(x),$$
where $p_{s,t}\colon M\times M\to\R$ is the transition probability density such that
\begin{equation*}
p_t(y)  =\int_M p_s(x) p_{s,t}(x,y) \dd\mu(x) \quad\text{and}\quad p_t(y) = p_{0,t}(x_0,y).
\end{equation*}
Furthermore, since $L$ is a time-homogeneous operator, we know that
$$p_{s,t}(x,y) = p_{0,t-s}(x,y) =: p_{t-s}(x,y).$$
From our assumption that $E$ is bracket-generating, it further follows from~\cite{StVa72} that $p_t(y) > 0$ and $p_{s,t}(x,y) >0$ for all $x,y\in M$ and all $s,t\in[0,T]$ with $s<t$. By definition, for any open set $U\subseteq M$ and for $s,t\in(0,T]$ with $s<t$, we have
$$\pp(X_t \in U) = \int_{U} p_t(y) \dd\mu(y)
\quad\text{and}\quad \pp(X_t \in U | X_s =x) = \int_U p_{s,t}(x,y) \dd\mu(y).$$

Let $L^*$ denote the adjoint of the operator $L$ with respect to the smooth volume density $\db \mu$, which is given by
\begin{displaymath}
    L^* = \Delta - 2Z - 2\dv_{\db\mu} Z.
\end{displaymath}
From the definition of the generator, we have the relation
$\partial_t \E[f(X_t)] = \frac{1}{2} \E[Lf(X_t)]$ for $f\in C^\infty(M)$, which in turn implies that, for $x,y\in M$ and $t\in[0,T]$,
\begin{displaymath}
    \left(\partial_t- \frac{1}{2} L^*\right) p_t(y) = 0, \quad \left(\partial_t - \frac{1}{2} L^{y,*}\right) p_{s,t}(x,y) =0
\end{displaymath}
as well as
\begin{equation}\label{eq:3rdidentity}
    \left(\partial_s + \frac{1}{2} L^x\right) p_{s,t}(x,y) = 0,
\end{equation}
subject to initial conditions given by $p_0(y) = \delta_{x_0}(y)$ and $p_{t,t} (x,y)= \delta_x(y)$. For the details, see e.g. Haussmann and Pardoux~\cite[p.~1191]{haussmann1986time}. 

\begin{remark}\label{remark:adjointprocess}
If we define $p_{s,t}^*(x,y) = p_{s,t}(y,x)$ and $P_{t}^*f(x) = \int_M f(y) p_{t}^*(x,y)\dd\mu(y)$ then by the above equations, we have $(\partial_t - \frac{1}{2}L^*) P_t f =0$ and $P_0 f = f$. It follows that as long as $\dv_{\db\mu} Z =0$ there then exists a stochastic process $(X_t^*)_{t\geq 0}$ with generator $\frac{1}{2} L^*$ and such that
$$p_{s,t}^*(x,y) = p_{t-s}^*(x,y) = p_{t-s}(y,x)$$
are its transition probability densities.
\end{remark}

\section{Bridge processes and time-reversed processes }
\label{sec:bridge}
We determine the generator for Riemannian and sub-Riemannian time-reversed diffusion bridge processes, with our main objective being to demonstrate how this generator is related to the score defined in \eqref{eq:score}.

\subsection{Bridge processes}
As before, let $(X_t)_{t\geq 0}$ on $(M,E,g,\db\mu)$ be the stochastic process with generator $\frac{1}{2}L=\frac{1}{2}\Delta +Z$ and initial value $x_0\in M$. For $T>0$ and $x_T\in M$, the diffusion bridge process $(Y_t)_{t\in[0,T]}$ on $(M,E,g,\db\mu)$ obtained by conditioning $(X_t)_{t\in[0,T]}$ on $X_T=x_T$ can formally be defined using disintegration of measure techniques, see e.g.~\cite{BMN21} and \cite{Hab18}. If there exists a unique path of minimal energy connecting the starting point $x_0$ to the final point $x_T$ then in small time, as shown in~\cite{BN22}, the laws of the diffusion bridge processes concentrate near the unique path of minimal energy. This result has been extended in~\cite{neel2020uniform} with $x_0$ and $x_T$ being connected by multiple minimal paths, subject to the additional assumption that none of them contains a segment which is so-called abnormal. We remark that this concentration property for the diffusion bridge processes in small time heavily relies on the underlying geometric structure and, as demonstrated in~\cite{Hab19}, may fail for more general diffusion bridge processes.

Using the formulation of the Doob $h$-transform, see e.g. \cite[Chapter~7.5]{sarkka2019applied}, an expression for the generator of the diffusion bridge process $(Y_t)_{t\in[0,T]}$ in terms of the logarithmic derivative of the transition probability densities $p_{t,T} = p_{T-t}$ and the generator $\frac{1}{2}L$ of the unconditioned process $(X_t)_{t\in[0,T]}$ can be obtained.
\begin{lemma} \label{lemma:Bridge}
The diffusion bridge process $(Y_t)_{t\in[0,T]}$ has the infinitesimal generator 
$$\frac{1}{2} L + \nabla^{x,E} \log p_{T-t}(\cdot , x_T).$$
\end{lemma}

\begin{proof}
    For $t\in[0,T)$, let the function $h_t\colon M\to\R$ be defined by $h_t(x) = p_{t,T}(x, x_T)$.
    According to~\eqref{eq:3rdidentity}, we know that
    \begin{equation}\label{eq:harmonic}
        \left(\partial_t + \frac{1}{2} L\right)h_t = 0,
    \end{equation}
    and by the Chapman–Kolmogorov equation, we further have, for $s\in(0,T-t)$,
    \begin{equation}\label{eq:CK}
        h_t(x) = \int_M p_{t,t+s}(x,y) h_{t+s}(y)\dd\mu(y).
    \end{equation}
    For any open set $U\subseteq M$, we obtain
    \begin{displaymath}
        \pp\left(X_{t+s}\in U|X_t=x,X_T=x_T\right)
        =\int_U\frac{p_{t,t+s}(x,y)p_{t+s,T}(y,x_T)}{p_{t,T}(x,x_T)}\dd\mu(y).
    \end{displaymath}
    Hence, the diffusion bridge process $(Y_t)_{t\in[0,T]}$ on $(M,E,g,\db\mu)$ has the same law as the Markov process $(X_t^h)_{t\in[0,T]}$ starting from $X_0^h = x_0$ whose transition kernel $p^h$ is given by, for $x,y\in M$,
    \begin{equation}\label{eq:newdensity}
        p_{t,t+s}^h(x,y) = p_{t,t+s}(x,y) \frac{h_{t+s}(y)}{h_t(x)}.
    \end{equation}
    It follows that, with $p_t^h(y) = p_{0,t}^h(x_0, y)$, we have
    $$\pp(Y_t \in U) = \int_U p_t^h(y) \dd\mu(y).$$
    This is indeed all well-defined due to $h_t(x)=p_{t,T}(x,x_T)>0$ for all $x\in M$ and all $t\in[0,T)$ 
    and, by the space-time regularity property~\eqref{eq:CK}, further defines a genuine transition probability density and thus a Markov process.
    The expression~\eqref{eq:newdensity} implies that, for $f\in C^\infty_0(M)$,
    \begin{displaymath}
        \mathbb{E}[f(X_{t+s}^h)|X_t^h =x]
        = \frac{\mathbb{E}[f(X_{t+s}) h_{t+s}(X_{t+s})| X_t =x]}{h_t(x)}.
    \end{displaymath}
    Using this formula, we can find the generator $\frac{1}{2} L^h_t$ for the Markov process $(X_t^h)_{t\in[0,T]}$, and hence for $(Y_t)_{t\in[0,T]}$, by first deducing that
    \begin{align*}
    \frac{1}{2}L^h_t f(x)
    & = \lim_{s \downarrow 0} \frac{\mathbb{E}[f(X_{t+s}^h)| X_t^h = x] -f(x)}{s} \\
    & = \lim_{s \downarrow 0} \frac{\mathbb{E}[f(X_{t+s})h_{t+s}(X_{t+s})| X_t = x] -f(x)h_t(x)}{s h_t(x)} \\ 
    & = \frac{(\partial_t+\frac{1}{2}L) (f h_t)(x)}{h_t(x)}.
    \end{align*}
    By applying~\eqref{eq:harmonic} and using $\frac{1}{2}L=\frac{1}{2}\Delta +Z$, we conclude that
    \begin{displaymath}
        \frac{(\partial_t+\frac{1}{2}L) (f h_t)}{h_t}
        =\frac{1}{2}Lf+ \frac{\langle \nabla^E f, \nabla^E h_t \rangle}{h_t},
    \end{displaymath}
    which establishes $\frac{1}{2}L^h_t = \frac{1}{2}L + \nabla^E \log h_t$. Due to $h_t(x) = p_{t,T}(x, x_T)=p_{T-t}(x,x_T)$, the claimed result follows.
\end{proof}

\subsection{Time-reversals} \label{sec:TimeReversals}
We consider a general process $(Y_t)_{t\geq 0}$ on $(M,E,g,\db\mu)$ with a time-dependent generator $\frac{1}{2}L_t = \frac{1}{2}\Delta + Z_t$ for a time-dependent vector field $Z_t$ and with initial value $x_0\in M$. We shall again assume non-explosion. For $T>0$, we denote the associated measure induced on the path space $C([0,T],M)$ by $\qq$. Let us further define $q_t$ on $M$ for $t\in [0,T]$ through, for any open set $U\subseteq M$,
\begin{equation}\label{eq:bridgedensity}
    \qq(Y_t \in U) = \int_U q_t(y) \dd \mu(y).
\end{equation}

We are now interested in looking at the time-reversed process $(\bar{Y}_t)_{t\in[0,T]}$ given by $\bar{Y}_t = Y_{T-t}$. We observe the following about its infinitesimal generator.
\begin{lemma} \label{lemma:Denoise}
The time-reversed processes $(\bar{Y}_t)_{t\in[0,T]}$ defined by setting $\bar{Y}_t = Y_{T-t}$ for $t\in[0,T]$ has infinitesimal generator
\begin{displaymath}
    \frac{1}{2}\bar{L}_t = \frac{1}{2}\Delta  - Z_{T-t} + \nabla^E \log q_{T-t}.
\end{displaymath}
Furthermore, we have that $\nabla^E q_t$ is in $L^2(E; g,\db\mu)$ for $t\in[0,T]$.
\end{lemma}
To produce our result, we modify the proof of~\cite[Theorem~2.1]{haussmann1986time} for a coordinate independent formulation.

\begin{proof}
In order to prove that the time-reversed process $(\bar{Y}_t)_{t\in[0,T]}$ has infinitesimal generator $\frac{1}{2}\bar{L}_t$, it suffices to show that, for $s,t\in [0,T]$ with $s<t$ and for any compactly supported function $f \in C^\infty_0(M)$, the mapping
$$t \mapsto f(\bar{Y}_t) - f(\bar{Y}_s) - \frac{1}{2} \int_s^t (\bar{L}_r f)(\bar{Y}_r) \dd r$$
defines a martingale.
As argued in~\cite{haussmann1986time}, the above is equivalent to the statement that, for all $f, \phi \in C^\infty_0(M)$,
$$\mathbb{E}\left[ \left(f(\bar{Y}_t) - f(\bar{Y}_s) - \frac{1}{2} \int_s^t (\bar{L}_r f)(\bar{Y}_r) \dd r \right) \phi(\bar{Y}_s)\right] = 0,$$
which after applying $\bar{Y}_t = Y_{T-t}$ and changing variables in the time integral amounts to
$$\mathbb{E}\left[ \left(f(Y_{T-t}) - f(Y_{T-s}) - \frac{1}{2} \int_{T-t}^{T-s} (\bar{L}_{T-r} f)(Y_r) \dd r \right) \phi(Y_{T-s})\right] = 0.$$
Since $T-t<T-s$ for $s<t$, it therefore suffices to establish that, for all $s,t\in [0,T]$ with $s<t$ and for all $f, \phi \in C^\infty_0(M)$,
$$\mathbb{E}\left[ \left(f(Y_{t}) - f(Y_{s})  + \frac{1}{2} \int_{s}^{t} (\bar{L}_{T-r} f)(Y_{r}) \dd r \right) \phi(Y_{t})\right] = 0.$$

Let $\qq_{s,t}$ with, for $x,y\in M$,
$$\db\qq_{s,t}(x,y)=q_s(x)q_{s,t}(x,y) \left(\db\mu \otimes \db\mu\right)(x,y)$$
denote the push-forward measure of the map $(x_{r})_{r \in [0,T]} \mapsto (x_s,x_t)$. We observe
\begin{align*}
& \E\left[ f(Y_s) \phi(Y_t) \right] = \int_{M \times M} f(x) \phi(y) \dd\qq_{s,t}(x,y) \\
& = \int_M f(x) q_s(x) \left( \int_M q_{s,t}(x,y) \phi(y) \dd\mu(y) \right) \db\mu(x) \\
& =: \int_M f(x) q_s(x) \Phi_{s,t}(x) \dd\mu(x) = \E\left[ f(Y_s) \Phi_{s,t}(Y_s) \right],
\end{align*}
where $\Phi_{s,t}(x) = \int_M q_{s,t}(x,y) \phi(y) \dd\mu(y) = \E[\phi(Y_t)| Y_s = x]$.
By using the properties that $(\partial_r + \frac{1}{2}L_r) \Phi_{r,t} =0$ and $(\partial_r - \frac{1}{2}L_r^*) q_r =0$, we further obtain
\begin{align*}
& \E[f(Y_t) \Phi_{t,t}(Y_t)] - \E[f(Y_s) \Phi_{s,t}(Y_s)]\\
&=\int_s^t\int_M f(x)(q_r(x)\partial_r\Phi_{r,t}(x) + \Phi_{r,t}(x)\partial_r q_r(x))\dd \mu(x)\dd r\\ 
&=\frac{1}{2}\int_s^t \int_M f(x)(-q_r(x)  L_r \Phi_{r,t}(x) + \Phi_{r,t}(x) L_r^* q_r(x))\dd\mu(x)\dd r\\
&=\frac{1}{2}\int_s^t \int_M  \Phi_{r,t}(x) (- L_r^*( f q_r)(x)   + f(x) L_r^* q_r(x))\dd\mu(x)\dd r.
\end{align*}
Since the adjoint $L^*_t$ of the operator $L_t=\Delta+2Z_t$ with respect to $\db\mu$ is given by
$$L^*_t = \Delta - 2Z_t - 2\dv_{\db\mu} Z_t,$$
we have
\begin{align*}
& \frac{1}{2}L^*_t(fq_t) - \frac{1}{2}fL^*_t q_t\\
& = \left(\frac{1}{2}\Delta  - Z_t - \dv_{\db\mu} Z_t\right) (fq_t) - f \left(\frac{1}{2}\Delta  - Z_t - \dv_{\db\mu} Z_t\right) q_t \\
& = q_t \left(\frac{1}{2}\Delta - Z_t\right) f +\langle \nabla^E f, \nabla^E q_t \rangle
= \frac{1}{2}q_t \bar{L}_{T-t} f.
\end{align*}
Hence, it follows that
\begin{align*}
& \E[f(Y_t) \phi(Y_t)] - \E[f(Y_s) \phi(Y_t)]=
  \E[f(Y_t) \Phi_{t,t}(Y_t)] - \E[f(Y_s) \Phi_{s,t}(Y_s)]\\
&= -\frac{1}{2}\int_s^t \int_M q_r(x) \Phi_{r,t}(x) (\bar{L}_{T-r}f)(x)\dd\mu(x)\dd r \\
&= -\frac{1}{2}\int_s^t \int_{M\times M} q_r(x) q_{r,t}(x,y) \phi(y) (\bar{L}_{T-r}f)(x) (\db\mu \otimes \db\mu)(x,y) \dd r \\
& = - \frac{1}{2}\E\left[ \phi(Y_t) \int_s^t (\bar{L}_{T-r}f)(Y_r) \dd r \right],
\end{align*}
as required.
\end{proof}

By combining Lemma~\ref{lemma:Bridge} and Lemma~\ref{lemma:Denoise}, we can determine the infinitesimal generator of the stochastic process obtained by time-reversing the diffusion bridge process corresponding to $(X_t)_{t\geq 0}$ with infinitesimal generator $\frac{1}{2} L = \frac{1}{2} \Delta + Z$ and starting point $x_0\in M$ conditioned on $X_T=x_T$ for $T>0$ and $x_T\in M$. Recall that we use $p_t(x,y) = p_{0,t}(x,y)$ for $x,y\in M$ to denote the transition probability density of the semi-martingale $(X_t)_{t\geq 0}$.
\begin{corollary}\label{cor:timereversedbridge}
    For $(Y_t)_{t\in[0,T]}$ the diffusion bridge process obtained by conditioning $(X_t)_{t\in[0,T]}$ with infinitesimal generator $\frac{1}{2} L = \frac{1}{2} \Delta + Z$ and starting point $x_0\in M$ on $X_T=x_T$, we define the time-reversed process $(\bar{Y}_{t})_{t\in[0,T]}$ by $\bar{Y}_{t} = Y_{T-t}$ for $t\in[0,T]$. Then the infinitesimal generator of $(\bar{Y}_{t})_{t\in[0,T]}$ is given by
    \begin{equation*}
        \frac{1}{2}\bar{L}_t
        =\frac{1}{2}\Delta-Z+S_{T-t}(x_0,\cdot).
    \end{equation*}
\end{corollary}

\begin{proof}
    By Lemma~\ref{lemma:Bridge}, the infinitesimal generator of $(Y_t)_{t\in[0,T]}$ takes the form
    \begin{displaymath}
        \frac{1}{2} \Delta +Z + \nabla^{x,E} \log p_{T-t}(\cdot , x_T).
    \end{displaymath}
    Lemma~\ref{lemma:Denoise} then implies that the time-reversed process $(\bar{Y}_{t})_{t\in[0,T]}$ has infinitesimal generator
    \begin{equation}\label{eq:timerevbridgegen}
        \frac{1}{2}\bar{L}_t
        =\frac{1}{2}\Delta-Z-\nabla^{x,E}\log p_{t}(\cdot,x_T)
        +\nabla^E\log q_{T-t},
    \end{equation}
    where $q_t$ is still defined by~\eqref{eq:bridgedensity}. It remains to show that the linear term reduces to the claimed expression. Since $(Y_t)_{t\in[0,T]}$ is the diffusion bridge process associated with $(X_t)_{t\geq 0}$ from $x_0$ to $x_T$ in time $T>0$, we have, for $y\in M$,
    \begin{displaymath}
        q_{T-t}(y)
        =\frac{p_{0,T-t}(x_0,y)p_{T-t,T}(y,x_T)}{p_{0,T}(x_0,x_T)}
        =\frac{p_{T-t}(x_0,y)p_{t}(y,x_T)}{p_{T}(x_0,x_T)}.
    \end{displaymath}
    It follows that
    \begin{displaymath}
        \nabla^E\log q_{T-t}
        =\nabla^{y,E} \log p_{T-t}(x_0,\cdot)+\nabla^{x,E} \log p_{t}(\cdot,x_T).
    \end{displaymath}
    This combined with~\eqref{eq:timerevbridgegen} and the definition \eqref{eq:score} establishes
    \begin{displaymath}
        \frac{1}{2}\bar{L}_t
        =\frac{1}{2}\Delta-Z+\nabla^{y,E} \log p_{T-t}(x_0,\cdot)
        = \frac{1}{2} \Delta - Z + S_{T-t}(x_0,\cdot),
    \end{displaymath}
    as required.
\end{proof}

\begin{remark}[Generative models]
We use the reversed generator from Lemma~\ref{lemma:Denoise} for simulating bridge processes. However, we could have also applied it for generative modelling as is often the focus in the machine learning literature on diffusion models, i.e., given a stochastic process $(Y_t)_{t\in[0,T]}$ with some final distribution $Y_T$, reversing the process enables sampling from the initial distribution for $Y_0$. Some examples of generative modelling on manifolds include \cite{bortoliRiemannianScoreBasedGenerative2022,huangRiemannianDiffusionModels2022}.
\end{remark}

\section{Approximating the score with neural networks}
\label{sec:score}
We now build the setup for approximating the score using neural networks in the sub-Riemannian setting. 
The ideas follow the ones presented in~\cite{HdBDT21}, where diffusion processes on $\R^d$ with uniformly positive definite diffusivity matrices are considered. The complexities arising in our context correctly handle the sub-Riemannian geometry, particularly the hypoellipticity of the sub-Riemannian diffusion processes.

In this section, we discuss loss functions involving the divergence of the score and a basic Euler--Maruyama sampling scheme leading to the score learning algorithm. We also give some details on the sampling scheme, for further use in Section~\ref{sec:denoising}, where we discuss the alternative denoising loss.

As previously, let $(E,g)$ be a bracket-generating sub-Riemannian structure on a manifold $M$ and let $\Delta$ be the sub-Laplacian with respect to the volume form $\db\mu$. We consider the stochastic process $(X_t)_{t\geq 0}=(X_t^{x_0})_{t\geq 0}$ with infinitesimal generator $\frac{1}{2} L = \frac{1}{2} \Delta + Z$ started from the initial value $x_0\in M$ and assumed to not explode. Let $p_t(x_0, \cdot)$ for $t>0$ denote the density of $(X_t)_{t\geq 0}$ with respect to $\db\mu$ whose score $S_t(x_0,\cdot)$ is given by \eqref{eq:score}.

Recall that according to Lemma~\ref{lemma:Bridge}, the bridge process associated with $(X_t)_{t\geq 0}$ from $x_0$ to $x_T\in M$ in time $T>0$ has generator $\frac{1}{2} L + \nabla^{x,E} \log p_{T-t}(\cdot , x_T)$. Since, if $Z = 0$, the operator $L^* = L = \Delta$ is symmetric, it follows from Remark~\ref{remark:adjointprocess} that then the generator of the bridge process can be expressed as $\frac{1}{2}L + S_{T-t}(x_T, \cdot)$. This would allow us to simulate bridge processes by approximating the score $S_{T-t}(x_T, \cdot)$ from the target point $x_T$.
However, in general, we have $p_t^*(y,x) = p_t(x,y)$ with $p_t^*(y,\cdot)$ a fundamental solution of $\partial_t - \frac{1}{2}L^*$ where the adjoint operator
\begin{displaymath}
    L^* = \Delta - 2Z - 2\dv_{\db\mu} Z
\end{displaymath}
with respect to $\db\mu$ may have a non-zero zeroth-order term. To avoid this problem, we instead consider the reversed bridge process $(\bar{Y}_t)_{t\in[0,T]}$ as in Corollary~\ref{cor:timereversedbridge} which has the generator
$$\frac{1}{2} \bar{L}_{t} = \frac{1}{2} \Delta-Z + S_{T-t}(x_0, \cdot ).$$
Thus, we will first need to determine the score from the starting point $x_0$ and then simulate bridges backwards from the endpoint $x_T$.

We henceforth write $S_{t}(\cdot) = S_t(x_0, \cdot)$.

\subsection{Score representation and loss function}
For estimating the score $S$, we let $S^\theta$ be a time-dependent vector field in $E$ given as the output of a neural network whose network parameters are denoted by $\theta$. If $(\sigma_1, \dots, \sigma_k)$ is frame for $E$, we can consider this problem as searching for functions, for $j\in\{1,\dots,k\}$,
\begin{align*}
    S^{\theta,j}\colon[0,T]\times M&\to\mathbb{R}\\
    (t,y) &\mapsto S_t^{\theta,j}(y)
\end{align*}
such that $S^{\theta} = \sum_{j=1}^k S^{\theta,j} \sigma_j$. Further details of the network architecture are given in Section~\ref{sec:experiments}.

When modifying the neural network weights to make the network match the score, we work with the squared $L^2$-distance as the loss function
\begin{align*}
    \scrE(\theta)  & =\int_0^T \int_{M} \left\| S^{\theta}_t(y)- \nabla^{y,E} \log p_{t}(x_0,y)\right\|_g^2\, p_t(x_0,y) \dd \mu(y) \dd t \\
    & = \int_0^T \mathbb{E}^{x_0}\left[\left\| S^{\theta}_t(X_t)- \nabla^{y,E} \log p_{t}(x_0,X_t)\right\|_g^2 \right] \dd t.
\end{align*}
For short-time approximations and for being able to compute the loss at every step, we will subdivide the energy $\scrE(\theta)$ into shorter time-intervals.
Observe that, for $s\in[0,t)$,
\begin{align*}
& \int_M \langle S_t^{\theta}(y) , \nabla^{y,E} \log p_{t}(x_0,y) \rangle_g \, p_t(x_0, y)\dd\mu(y) =  \int_M \langle S_t^{\theta}(y) , \nabla^{y,E} p_{t}(x_0,y) \rangle_g \dd\mu(y) \\
& = \int_{M \times M} \langle S_t^{\theta}(y) , \nabla^{y,E} \log p_{t-s}(z,y) \rangle_g \, p_{s}(x_0,z) p_{t-s}(z,y) (\db\mu \otimes \db\mu)(z,y)  \\
& =\int_{M \times M} \langle S_t^{\theta}(y) , \nabla^{y,E} \log p_{t-s}(z,y) \rangle_g \dd \mathbb{P}_{s,t}(z,y)  = \E^{x_0}\left[\langle S^\theta_t(X_t), S_{t-s}(X_s, X_t) \rangle_g \right].
\end{align*}
From these computations, we deduce that for $e_{s,t}(\theta)$ defined by
\begin{align} \label{est} 
    e_{s,t}(\theta) &  = \E^{x_0}\left[\langle S^\theta_t(X_t), S^\theta_t(X_t) - 2 S_{t-s}(X_s, X_t) \rangle_g \right],
\end{align}
the expression
$$\mathbb{E}^{x_0}\left[\left\| S^{\theta}_t(X_t)- \nabla^{y,E} \log p_{t}(x_0,X_t)\right\|_g^2 \right] -e_{s,t}(\theta)$$
is constant with respect to $\theta$. This implies in particular that, for any subdivision $0=t_0 < t_{1} < \cdots < t_n= T$ of the interval $[0,T]$, there exists a constant $C$ such that
$$\scrE(\theta) = \scrE_{0,T}(\theta) + C = \sum_{i=0}^{n-1} \scrE_{t_{i},t_{i+1}}(\theta) + C,
\quad\text{where}\quad \scrE_{s,t}(\theta) := \int_s^t e_{s,r}(\theta) \dd r.$$
It follows that letting our neural network minimise relative to $\scrE(\theta)$ is equivalent to minimising with respect to $\scrE_{0,T}(\theta)$. Note that the definition for $\scrE_{s,t}(\theta)$ still contains the expression of the true score, which is usually intractable. In the subsequent sections, we discuss how we can consider minima of the loss functions without knowing the true score explicitly.

\subsection{Loss function with divergence}
By exploiting an integration-by-parts trick, the loss function can be computed without explicitly knowing the true score. We prove the result here for the sub-Riemannian case that has been previously derived for the Euclidean setting in~\cite{hyvarinenEstimationNonNormalizedStatistical2005} and for Riemannian manifolds in~\cite{bortoliRiemannianScoreBasedGenerative2022}, based on earlier work in \cite{song2019generative,song2020improved}.

\begin{theorem}\label{thm:IBP}
For any $s,t\in[0,T]$ with $s<t$ and for $e_{s,t}(\theta)$ defined as in \eqref{est}, we have
\begin{equation}
\begin{split}
    e_{s,t}(\theta) & = \int_{M} \left( \left\|S^\theta_{t}(y)\right\|^2_g + 2(\dv_{\db\mu} S_{t}^\theta)(y) \right) \db\pp_{t}(y) \\
    & = \E^{x_0}\left[\left\|S^\theta_{t}(X_t)\right\|^2_g + 2(\dv_{\db\mu} S_{t}^\theta)(X_t)\right]  =: e_t^{(2)}(\theta).
\end{split}
\label{eq:div_loss}
\end{equation}
\end{theorem}
We remark that this rewritten expression for $e_t^{(2)}(\theta)$ does not depend explicitly on the score itself. Since then $\scrE_{s,t}(\theta) = \int_s^t e_r^{(2)}(\theta) \dd r$, we can use the formulation of Theorem~\ref{thm:IBP} to minimise $\scrE_{0,T}(\theta)$ without knowing the score.
In computations, the expectations are approximated by averaging over samples $X^{(1)},\dots,X^{(K)}$ from the stochastic process $X=(X_t)_{t\in [0,T]}$ through, with $\delta = \frac{T}{n}$,
\begin{equation}\label{eq:loss_samples}
\scrE_{0,T}(\theta)
\approx
\frac{\delta}{K}\sum_{l=1}^K \sum_{i=1}^n \left(\left\|S^\theta_{i\delta}(X_{i\delta}^{(l)})\right\|^2_g + 2(\dv_{\db\mu} S_{i\delta}^\theta)(X_{i\delta}^{(l)})\right).
\end{equation}

\begin{proof}[Proof of Theorem~\ref{thm:IBP}]
Recall that on the manifold $M$ without boundary we have, for $f\in C^\infty(M)$ and a smooth vector field $V$ on $M$,
\begin{equation}\label{eq:divthm}
    \int_M \db f(V) \dd\mu = -\int_M f \dv_{\db\mu}(V) \dd\mu.
\end{equation}
We apply~\eqref{eq:divthm} to deduce
\begin{align*}
& \int_{M \times M} \langle S_t^\theta(y), \nabla^{y,E} \log p_{t-s}(z,y) \rangle_g \dd\pp_{s,t}(z,y) \\
& = \int_{M \times M} p_s(x_0, z) (\db^y p_{t-s}(z,y)(S_t^\theta(y))) \, (\db\mu \otimes \db\mu)(z,y) \\
& = - \int_{M \times M} p_s(x_0,z) p_{t-s}(z,y)(\dv_{\db\mu} S_t^\theta)(y) \, (\db\mu \otimes \db\mu)(z,y)  \\
& = - \int_{M} (\dv_{\db\mu} S_t^\theta)(y) \dd\pp_t(y).
\end{align*}
The claimed result that $e_{s,t}(\theta) = e^{(2)}_t(\theta)$ then follows.
\end{proof}

\subsubsection{Expression in local coordinates} \label{sec:localSDE}
Suppose we choose a local coordinate system $(x^1, \dots, x^d)$ on $M$ and let $\db x = \db x^1 \dots \db x^d$ denote the Euclidean volume measure in these coordinates. Consider a second order operator $L = \Delta + 2Z$ on $M$ where $\Delta$ is a sub-Laplacian of a sub-Riemannian structure $(E,g)$. By Remark~\ref{re:ChoiceMeasure}, we know that with the appropriate choice for $Z$ we can  assume that $\Delta = \Delta_{E,g,\db x}$ is the sub-Laplacian with respect to $\db x$. Let $(\sigma_1, \dots, \sigma_k)$ be a local orthonormal frame for $(E,g)$, which is described by a $d \times k$ matrix $(\sigma_j^i)$ such that $\sigma_j = \sum_{i=1}^d \sigma_j^i \partial_{x^i}$ for $j\in\{1,\dots,k\}$. With $Z = \sum_{i=1}^d Z^i \partial_{x^i}$, we have
\begin{align*} L = \Delta +2Z &= \sum_{j=1}^k \sum_{i,l=1}^d \sigma_j^l \sigma^i_j \partial_{x^l}  \partial_{x^i} + \sum_{j=1}^k \left( \sum_{l=1}^d \partial_{x^l} \sigma_j^l \right) \sum_{i=1}^d \sigma_j^i \partial_{x^i} \\
& \qquad + \sum_{j=1}^k \sum_{i,l=1}^d \sigma^l_j   \partial_{x^l} \sigma_j^i \partial_{x^i}  + 2\sum_{i=1}^d Z^i \partial_{x^i}. \end{align*}

Let $(W_t)_{t\geq 0}$ be a standard Brownian motion on $\R^k$ and set
$$\sigma_0 =  \sum_{j=1}^k (\dv_{\db x} \sigma_j) \sigma_j + 2Z = \sum_{j=1}^k \left( \sum_{l=1}^d \partial_{x^l} \sigma_j^l \right) \sum_{i=1}^d \sigma_j^i \partial_{x^i}  + 2\sum_{i=1}^d Z^i \partial_{x^i}$$
as well as $\tau_0 = \sigma_0 + \frac{1}{2} \sum_{j=1}^k (\bar{\nabla}_{\sigma_j} \sigma_j)$, where
$\bar{\nabla}$ denotes the flat connection given by
\begin{equation} \label{FlatConnection}
\bar{\nabla}_{V_1} V_2 = \sum_{i=1}^d (V_1 V_2^i) \partial_{x^i} = \sum_{i,l=1}^d V_1^{l} (\partial_{x^{l}} V_2^i) \partial_{x^i} .
\end{equation}
Then the stochastic process $(X_t)_{t\geq 0}$ with generator $\frac{1}{2} L$ is the unique strong solution to the Stratonovich stochastic differential equation
\begin{equation} \label{Stratonovich} \db X_t = \sum_{j=1}^k \sigma_j(X_t) \circ \db W_t^j + \sigma_0(X_t) \dd t,\qquad X_0=x_0, \end{equation}
or to the Itô stochastic differential equation
\begin{equation} \label{Ito} \db X_t = \sum_{j=1}^k \sigma_j(X_t) \dd W_t^j + \tau_0(X_t) \dd t ,\qquad X_0=x_0.\end{equation}
Writing the output of the neural network as $S^\theta = \sum_{j=1}^k S^{\theta,j}\sigma_j$, we obtain
\begin{align*}
    e_t^{(2)}\left(\theta\right) &=  \int_M \left( \sum_{j=1}^k S_t^{\theta,j} (y)^2 \right) p_t(x_0,y) \dd y \\
&\quad + 2 \int_M \left(  \sum_{j=1}^k 
\sum_{i=1}^d\left(S_t^{\theta,j}(y) (\partial_{x^i} \sigma^i_j)(y) + \sigma_j^i(y) (\partial_{x^i} S_t^{\theta,j})(y) \right)\right) p_t(x_0,y) \dd y.
\end{align*}
If we now have sample paths $X^{(1)}$, $\dots$, $X^{(K)}$ and discretise the interval $[0,T]$ with $\delta = \frac{T}{n}$, then \eqref{eq:loss_samples} gives rise to
\begin{align*}
    \scrE_{0,T}(\theta) & \approx \frac{\delta}{K} \sum_{l=1}^K \sum_{i=1}^n  \sum_{j=1}^k S^{\theta, j}_{i \delta}(X_{i\delta}^{(l)})^2 \\
& \quad + \frac{\delta}{K} \sum_{l=1}^K \sum_{i=1}^n  \sum_{j=1}^k  
\sum_{m=1}^d\left( S_{i\delta}^{\theta,j}(X_{i\delta}^{(l)}) (\partial_{x^m} \sigma^m_j)(X_{i \delta}^{(l)}) + \sigma_j^m(X_{i\delta}^{(l)}) (\partial_{x^m} S_{i\delta}^{\theta,j})(X_{i\delta}^{(l)}) \right).
\end{align*}

\subsection{Time integration}
\label{sec:integration}
We now explicitly describe an Euler--Maruyama simulation scheme for finding paths of the stochastic process $(X_t)_{t\in[0,T]}=(X_t^{x_0})_{t\in[0,T]}$ with generator $\frac{1}{2} L = \frac{1}{2} \Delta+Z$. Later on, we will further exploit this approximation for the denoising version of the loss function. 

In order to describe the method, we assume that $M = \mathbb{R}^d$ topologically.
If the manifold $M$ has a different topology, we can either work in local coordinates or use geometric Euler--Maruyama methods. For instance, see~\cite{piggott2016geometric,muniz2022higher} for some examples of methods for matrix Lie groups.

Assume that $(X_t)_{t\in[0,T]}$ can, at least locally, be considered as the solution of \eqref{Ito} in a coordinate system $(x^1,\dots,x^d)$. For a fixed $n\in\N$, we set $\delta = \frac{T}{n}$ and $t_i = i \delta$ for $i\in\{0,1,2,\dots, n\}$. The Euler--Maruyama approximation $(\hat X_t)_{t\in[0,T]} = (\hat X_t^{n})_{t\in[0,T]}$ is then defined by, for $i\in\{0,1,2,\dots, n-1\}$ and $t\in(0,\delta]$,
\begin{equation}
    \hat X_{t_i+t} = \hat X_{t_i} + \sum_{j=1}^k \sigma_j(\hat X_{t_{i}})  (W_{t_i+t}^j-W_{t_i}^j) + \soo(\hat X_{t_{i}}) t, \qquad \hat X_0 = x_0.
    \label{eq:EulerMaruyama}
\end{equation}
The Euler--Maruyama approximation $(\hat X_t^n)_{t\in[0,T]}$ converges strongly to $(X_t)_{t\in[0,T]}$ as $n\to\infty$ under appropriate growth conditions on the vector fields $\soo,\sigma_1,\dots, \sigma_k$, see \cite[Theorem~4.5.3 and Theorem~9.6.2]{kloeden1992approximation} for details.

\subsection{Score estimation}
Together the loss $e_{s,t}(\theta)$ given by~\eqref{eq:div_loss} and the sampling scheme \eqref{eq:EulerMaruyama} provide the tools needed to train the score approximation. The explicit algorithm is listed in Algorithm~\ref{algo:Divergence} below.
\begin{algorithm}[H] 
\caption{Estimating the score using the divergence loss}
\label{algo:Divergence}
\begin{algorithmic}[1]
\Require{
\Statex
\begin{enumerate}[$\bullet$]
        \item \textit{Initial set-up:} Initial point $x_0$, final point $x_T$.
      \item \textit{Diffusion vector fields:} Vector fields $\sigma_1, \dots, \sigma_k$ through the function with values in $d\times k$ matrix $(\sigma_j^i)$.
      \item \textit{Learning setup:} Learning rate $\epsilon>0$, number of iterations $N\in\mathbb{N}$.
  \end{enumerate}
} 
\State \textit{Initialisation}: Initialise weights $\theta$ randomly.
\For{$i = 1$ to $N$} 
    \State \textit{Forward pass}: Draw batch of $K\in\mathbb N$ Euclidean Wiener process samples $W^{(1)},\dots,W^{(K)}$ and integrate~\eqref{eq:EulerMaruyama} to obtain sample paths $X^{(1)},\dots,X^{(K)}$.
    \State \textit{Compute Loss}: Evaluate the approximation~\eqref{eq:loss_samples} of $\scrE_{0,T}(\theta)$ using samples $X^{(1)},\dots,X^{(K)}$.
    \State \textit{Backward pass}: Compute the gradient $\nabla_\theta\scrE_{0,T}(\theta)$.
    \State \textit{Update Parameters}: Set $\theta\leftarrow \theta-\epsilon\nabla_\theta\mathcal \scrE_{0,T}(\theta)$.
\EndFor
\end{algorithmic}
\end{algorithm}

\section{Approximating loss functions for short time steps}
\label{sec:denoising}
Even though the loss function in $\scrE_{s,t}(\theta) = \int_s^t e_r^{(2)}(\theta) \dd r$ defined by~\eqref{eq:div_loss} has no explicit dependency on the score, using the term $(\dv_{\db\mu} S_t^\theta)p_t$ can in practice be problematic because the integral often is approximated with finite samples at the steps of the training, and the training might therefore at each step minimise the divergence at the sample points only, resulting in unstable convergence. Moreover, the computational expense of finding derivatives of the neural network and taking gradients of those for the optimisation can be a problem if the dimension of $M$ is high. These issues have also been observed in the Euclidean case, and a common way to deal with these problems is the denoising loss, see \cite{hyvarinenEstimationNonNormalizedStatistical2005,vincentConnectionScoreMatching2011}. In this section, we show how this approach can be generalised to the sub-Riemannian setting.

With the notation in this paper, the denoising loss arises from \eqref{est} using explicit approximations of the true score for $s$ and $t$ close, e.g. for steps from $t_{i}$ to $t_{i+1}$. Such steps can be approximated, in the Euclidean setting, by normal distributions and, in the Riemannian situation, by heat kernel approximations or tangent space normal distributions as in \cite{bortoliRiemannianScoreBasedGenerative2022}. In the sub-Riemannian context, we take horizontal derivatives of the score, but we also need to take account of the change of the distribution between tangent spaces and the hypoellipticity of the diffusion process. We therefore look at short-time approximations for the steps of diffusion processes on sub-Riemannian manifolds.

\subsection{Short-time score approximations with Euler steps}  \label{sec:EMapprox}
We choose a local coordinate system $(x^1, \dots, x^d)$ such that we can identify $M$ with $\R^d$. As in Section~\ref{sec:integration}, let us consider the approximation $(\hat X_t)_{t\in[0,T]}$ given as the solution to \eqref{eq:EulerMaruyama} from the Euler--Maruyama integration. The equation is stated using the global orthonormal frame $(\sigma_1,\dots,\sigma_k)$ represented by the matrix $\sigma = (\sigma^i_j)$. We write $\Sigma = \sigma \sigma^\top$ and observe, for $t\in(0,\delta)$, that $\hat X_t \sim N(t \soo(x_0), t \Sigma(x_0))$ which is supported on $t \soo(x_0) + E_{x_0}$. It follows that the probability measure $\hat \pp_t =\hat \pp_t^n$ of $\hat X_t$ is not even absolutely continuous with respect to $\pp_t$ whenever $t\in (0,\delta)$ and $E\not= TM$. If we write $\db \hat \pp_t = \hat p_t \dd\mu$ then $\hat p_t$ cannot be described as a continuous function.

We further encounter the problem that in general the tangent space at the point $y \in t\soo(x_0) + E_{x_0}$ is not included in the horizontal distribution at~$y$, i.e., we generally have $T_y(t  \soo(x_0) + E_{x_0}) \not \subseteq E_y$. In particular, it becomes problematic to define a horizontal gradient. We instead consider the following solution. We take $S^\theta = \sum_{j=1}^k S^{\theta,j} \sigma_j$ with associated vector representation $\bfS^\theta = (S^{\theta,1}, \dots, S^{\theta,k})$ and observe that
\begin{align*}
e_{t_{i},t_i+t}(\theta) & = \E^{x_0}\left[\langle S^\theta_{t_i+t}(X_{t_i+t}), S^\theta_{t_i+t}(X_{t_i+t}) - 2 S_{t}(X_{t_i}, X_{t_i+t}) \rangle_g \right] \\
& =  \sum_{j=1}^k \int_{M\times M} \left( S^{\theta,j}_{t_i+t}(y)^2 - 2 S^{\theta,j}_{t_i+t}(y) \dd^y \log p_{t}(z,y)(\sigma_j(y))  \right) \db \mathbb{P}_{t_{i},t_i+t}(z,y).
\end{align*}
As a result of the discussed problems with the horizontal gradient in the approximation scheme due to only having access to derivatives of $\hat p_t(z,\cdot)$ in the $E_z$ directions, we evaluate the differential above at the initial point instead of the target point and use the approximation
\begin{equation*}
\hat e_{t_{i},t_i+t}(\theta) := \sum_{j=1}^k \int_{M\times M} \left( S^{\theta,j}_{t_i+t}(y)^2 - 2 S^{\theta,j}_{t_i+t}(y) \dd^y \log \hat p_{t}(z,y)(\sigma_j(z)) \right)
\db \hat {\mathbb{P}}_{t_{i},t_i +t}(z,y),
\end{equation*}
where $\hat \pp_{s,t}$ is the probability measure on $M \times M$ of $(\hat X_s, \hat X_t)$.
For $t\in(0,\delta)$, we have
$$\hat p_{t}(z,y) = \delta_{t\soo(z) + E_z}(y) \frac{1}{\sqrt{(2\pi t)^k \prod_i \lambda_i}} \exp\left(- \frac{\langle y-z- t\soo(z), y-z- t\soo(z) \rangle_{g(z)}}{2t}\right),$$
where $\prod_i \lambda_i$ is the product of all non-zero eigenvalues of $\Sigma(z)$. These functions have well-defined derivatives in the directions of $E_z$ which yields
$$\db^y \log \hat{p}_{t}(z,y)(\sigma_j(z)) = -\frac{1}{t} \langle \sigma_j(z) , y-z - t\soo(z) \rangle_{g(z)}.$$
Recalling now that $\hat X_{t_i+t} - \hat X_{t_i} = \sigma(\hat X_{t_{i}})  (W_{t_i+t}-W_{t_i}) + \soo(\hat X_{t_{i}}) t$, we obtain
\begin{align*}
\hat e_{t_{i},t_i+t}(\theta) 
& = \sum_{j=1}^k \E^{x_0} \left[ S^{\theta,j}_{t_i+t}(\hat X_{t_i+t})^2 + \frac{2}{t} \langle (S^{\theta,j}_{t_i+t} \sigma_j)(\hat X_{t_i}) , \hat X_{t_i+t} - \hat X_{t_i} - t\soo(\hat X_{t_i}) \rangle_{g(X_{t_i})}\right] \\
& =  \E^{x_0} \left[ \left\langle \bfS^{\theta}_{t_i+t}(\hat X_{t_i+t}), \bfS^{\theta}_{t_i+t}(\hat X_{t_i+t}) + \frac{2}{t} (W_{t_i+t}-W_{t_i}) \right\rangle_{\mathbb{R}^k}\right].
\end{align*}

In summary, if we approximate the diffusion process $(X_t)_{t\in[0,T]}$ with $(\hat X_t)_{t\in[0,T]}$ which relies on the Euler--Maruyama scheme, then we can use as a loss function $\hat \scrE_{0,T}(\theta) = \sum_{i=0}^{n-1} \int_{t_{i}}^{t_{i+1}} \hat e_{t_{i},s}(\theta) \dd s$ that tries to predict each step. However, since each Euler step merely has positive probability on a proper subspace, we can only predict the gradient of $\log \hat p_t( \hat X_{t_i}, \hat X_{t_i+t})$ at the initial point $\hat X_{t_i}$. We again use a time-discretisation with $n$ intervals and write $\delta = \frac{T}{n}$. We generate $K \times n$ random vectors $\Delta_i W^{(l)}$ for $i\in\{1, \dots, n\}$ as well as $l\in\{1, \dots, K\}$ that are drawn from a $N(0,\delta I)$-distribution in~$\mathbb{R}^k$. We define first $X^{(l)}_0 = x_0$ and then iteratively
$$\hat X_{i+1}^{(l)}= \hat X_{i}^{(l)}+\sum_{j=1}^k \sigma_j(\hat X_i^{(l)}) \Delta_{i+1} W^{(l),j}+ \delta \tau_0(X_i^{(l)})$$
and compute the loss function by
\begin{align} \label{eq:denoising_loss}
\hat \scrE_{0,T}(\theta) &\approx \frac{1}{K} \sum_{l=1}^K \sum_{i=1}^n  \left\langle \bfS^{\theta}_{i \delta}(\hat X_i^{(l)}), \delta \, \bfS^{\theta}_{i \delta}(\hat X_{i}^{(l)}) + 2 \Delta_i W^{(l)} \right\rangle_{\mathbb{R}^k} \\ \nonumber
& = \frac{\delta}{K} \sum_{l=1}^K \sum_{i=1}^n  \left\| \bfS^{\theta}_{i \delta}(\hat X_i^{(l)})+ \frac{1}{\delta} \Delta_i W^{(l)} \right\|_{\mathbb{R}^k}^2 +C.
\end{align}
We can use the loss $\hat \scrE_{0,T}(\theta)$ to train a neural network by applying Algorithm~\ref{algo:Divergence} with the loss $\scrE_{0,T}(\theta)$ replaced by $\hat\scrE_{0,T}(\theta)$.

\subsection{Taylor expansion and approximations of stochastic integrals}
\label{sec:TaylorApproximations}
We improve on the method in Section~\ref{sec:EMapprox} by considering a more complicated approximation scheme for each step. Contrary to what is the case for the Euler--Maruyama approximation, each step in our new approximation scheme gives us densities with respect to Lebesgue measure that are smooth and positive. As a tool in this new approximation, we use the stochastic Taylor expansion. We will describe the resulting approach in Section~\ref{sec:SimStep}.

We first look at the Stratonovich Taylor expansion for a stochastic process found in \cite[Chapter~5 and Chapter~10.7]{kloeden1992stochastic}.
Let $(X_t)_{t\geq 0}$ be the unique strong solution of the Stratonovich stochastic differential equation \eqref{Stratonovich} with vector fields $\sigma_0, \sigma_1, \dots, \sigma_k$. We use the convention $W^0_t =t$ and write $W_{s,t}^j = W_t^j - W_s^j$. For a multi-index $\alpha = (\alpha_1, \dots, \alpha_l) \in \{0,\dots, k\}^l$ for $l\in\N_0$, we set $l(\alpha) = l$ and we let $n(\alpha)$ denote the number of zeros in the multi-index $\alpha$. We further define the differential operators
$$\sigma_\alpha = \sigma_{\alpha_1} \sigma_{\alpha_{2}} \cdots \sigma_{\alpha_l}$$
as well as the stochastic processes given by
$$J^\alpha_{s,t} = \int_{s<t_1< \cdots < t_{l}<t} \circ \dd W_{t_1}^{\alpha_1} \circ \cdots \circ \db W_{t_l}^{\alpha_l}.$$
Let $(x^1, \dots, x^d)$ be a choice of local coordinate system for the manifold $M$ and define a function $\bfsigma_\alpha =(\sigma_{\alpha}^i)_i$ with values in $\R^d$ such that
$$\sigma_\alpha^i = \sigma_\alpha x^i,$$
i.e., the differential operator $\sigma_\alpha$ applied to the $i$-th coordinate function.
For further reference, we state below~\cite[Theorem~10.7.1 and Corollary~10.7.2]{kloeden1992stochastic}.

\begin{lemma}[Stratonovich Stochastic Taylor expansion]\label{lem:stratstochtaylor}
Let $\lfloor \cdot \rfloor$ denote the floor function, i.e., $\lfloor a \rfloor = \max_{n \in \mathbb{Z}, n \leq a } n$. Let $\calA_\gamma$ be the collection of all multi-indices $\alpha$ satisfying $l(\alpha) +n(\alpha) \leq 2 \gamma$, and set $\calA_\gamma^+ = \{ (\beta_1, \alpha) : \alpha \in \calA_\gamma, \beta_1 \in\{0,1,\dots, k\}\}$. We say that a vector-valued function $\mathbf{f}$ satisfies \eqref{Halpha} if almost surely
\begin{equation} \tag{$H_\alpha$} \label{Halpha}
\int_{0}^T \left\| \int_{0<t_1< \cdots <t_{l-1}<s} \mathbf{f}(X_{t_1}) \dd W_{t_1}^{\alpha_1} \cdots \dd W_{t_{l-1}}^{\alpha_{l-1}}\right\|^{2-\delta_{0,\alpha_l}} \dd s < \infty.
\end{equation}
Let integers $\gamma$ and $n$ be given such that $\gamma\geq 0$ and $n>T$. Set $\delta = \frac{T}{n}$ and $t_i = i \delta$ for $i\in\{0,1, \dots,n\}$, and define $\hat X^i$ iteratively by $\hat X^0 = 0$ as well as
$$\hat X^{i+1} = \sum_{l(\alpha) + n(\alpha)\leq 2\gamma} J_{t_{i},t_{i+1}}^\alpha \cdot \bfsigma_{\alpha}(\hat X^{i}),$$
which we extend to $(\hat X_t)_{t\in[0,T]}$ through
$$\hat X_{t_i+t} = \sum_{l(\alpha) + n(\alpha)\leq 2\gamma} J_{t_{i},t_{i}+t}^\alpha \cdot \bfsigma_{\alpha}(\hat X^{i}) , \qquad for\enspace t\in [0, \delta].$$
If we assume that we have constants $C_1, C_2$ such that both
\begin{itemize}
\item $\| \bfsigma_\alpha(x) - \bfsigma_\alpha(y)\|_{\R^d} \leq C_1 \|x-y\|_{\R^d}$ for any $\alpha \in \calA_\gamma$, and
\item $\bfsigma_\alpha$ is $C^{1,1}$, satisfies condition \eqref{Halpha} and furthermore the inequality
$$\| \bfsigma_{\alpha}(x) \|_{\R^d} \leq C_2(1+\|x\|_{\R^d})$$
for any $\alpha \in \calA_\gamma \cup \calA_\gamma^+$,
\end{itemize}
then we have the following bound, for some constant $C$,
$$\E^{x_0}\left(\sup_{0 \leq t \leq T} \| X_t - \hat X_t\|_{\R^d} \right) \leq C \delta^\gamma.$$
\end{lemma}

Let us now consider the result of Lemma~\ref{lem:stratstochtaylor} for the particular case where $\gamma=1$. For $t \in [t_i, t_{i+1}]$, we obtain
\begin{align*}
\hat X_t &= \hat X_{t_{i}} + (t-t_i) \bfsigma_0(\hat X_{t_{i}}) + \sum_{j=1}^k W^j_{t_i,t} \bfsigma_j(\hat X_{t_{i}})  + \sum_{j,l=1}^k  \int_{t_{i}}^{t} W^j_{t_{i},s} \circ \db W_s^l \bfsigma_{(j,l)}(\hat X_{t_{i}})  \\
&= \hat X_{t_{i}} + (t-t_i) \bfsigma_0(\hat X_{t_{i}}) + \sum_{j=1}^k W^j_{t_i,t} \bfsigma_j(\hat X_{t_{i}})  \\
& \qquad + \frac{1}{2} \sum_{j,l=1}^k  W_{t_{i},t}^j W_{t_i,t}^{l} \bfsigma_{(j,l)}(\hat X_{t_{i}})   + \sum_{1\leq j<l\leq k} A^{j,l}_{t_i,t} (\bfsigma_{(j,l)}-\bfsigma_{(l,j)})(\hat X_{t_{i}}), 
\end{align*}
where $A_{s,t}^{j,l} = \frac{1}{2} \int_s^t (W_{s,r}^{j} \circ \db W^{l}_r-  W_{s,r}^{l} \circ \db W^{j}_r)$ is the so-called Lévy area.

We continue by discussing several methods for approximating this Lévy area. One option, see e.g.~\cite{kloeden1992approximation,milstein2013numerical,mrongowius2022approximation} exploits the Fourier expansion of the involved Brownian motions and eliminates a dependency to give, for $j,l\in\{1,\dots,k\}$,
\begin{equation} \label{LevyArea}
    A^{j,l}_{t,t+ h} = 
    \sum_{m=1}^\infty\left(a_{l,m} W^j_{t,t+h} - a_{j,m} W^l_{t,t+h} \right)+  \pi \sum_{m=1}^\infty m(a_{j,m} b_{l,m} -a_{l,m} b_{j,m} )
\end{equation}
in terms of independent Gaussian random variables with $W^j_{t,t+h} \sim N(0, h)$ and
\begin{align*}
    a_{j,m}
    &=\frac{2}{h} \int_0^h \left(W_s^j -\frac{s}{h} W_h^j \right)
    \cos\left(\frac{2\pi ms}{h}\right)\dd s
    \sim N\left(0, \frac{h}{2\pi^2 m^2}\right),\\
    b_{j,m}
    &=\frac{2}{h} \int_0^h \left(W_s^j -\frac{s}{h} W_h^j \right)
    \sin\left(\frac{2\pi ms}{h}\right)\dd s
    \sim N\left(0, \frac{h}{2\pi^2 m^2}\right).
\end{align*}
We can then approximate the Lévy area by truncating the series in \eqref{LevyArea}.

A second option which only requires the simulation of independent Gaussian random variables makes use of the polynomial decomposition of Brownian motion developed in~\cite{foster20,semicircle} and results in the expansion, see~\cite{foster_habermann,KuznetsovLevyArea1},
\begin{equation}\label{eq:polyapprox}
    A_{t,t+h}^{j,l}
    =\frac{1}{2}\left(c_{l,1} W_{t,t+h}^j - c_{j,1} W_{t,t+h}^l \right)
    +\frac{1}{2}\sum_{m=1}^\infty\left(c_{j,m}c_{l,m+1}-c_{j,m+1}c_{l,m}\right),
\end{equation}
with independent Gaussian random variables given by $W_{t,t+h}^j\sim N(0,h)$ as well as
\begin{displaymath}
    c_{j,m}\sim N\left(0,\frac{h}{2m+1}\right).
\end{displaymath}
Truncating the series in~\eqref{eq:polyapprox} then again yields an approximation for the Lévy area.

If we have a discretisation $0 =t_0 < \cdots < t_n=T$ with $\delta = t_{i+1}- t_i$ then we can obtain the bound in Lemma~\ref{lem:stratstochtaylor} with approximated Lévy area under appropriate bounds. If $\hat A_{t,t+h}^{j,l}$ is an approximation for the Lévy area $A_{t,t+h}^{j,l}$ such that $\E^{x_0}(\sup_{t_i \leq t \leq t_{i+1}}| A^{j,l}_{t_{i},t} - \hat A^{j,l}_{t_i,t}| ) \leq C_0 \delta^3$ on each interval $[t_i,t_{i+1}]$, and if we replace $ A^{j,l}_{t_i,t}$ by $\hat A^{j,l}_{t_i,t}$ in the definition of $(\hat X_t)_{t\in[0,T]}$, then by \cite[Corollary~10.7.3]{kloeden1992stochastic}, we still have
$$\E^{x_0}\left(\sup_{0 \leq t \leq T} \| X_t - \hat X_t\|_{\R^d} \right) \leq C \delta.$$

\begin{remark}
Since we are working with local coordinates, we can identify vectors $v = \sum_{i=1}^d v^i \partial_{x^i}$ with the point $(v^1,\dots, v^d)$. Under this correspondence, we can identify $\bfsigma_j$ with $\sigma_j$ and $\bfsigma_{(j,l)}$ with $\bar{\nabla}_{\sigma_j} \sigma_l$ for $\bar\nabla$ defined by \eqref{FlatConnection}. It follows that $\hat X_t$ for $t \in [t_i, t_{i+1}]$ can be expressed as
\begin{align*}
\hat X_t 
&= \hat X_{t_{i}} + (t-t_i) \sigma_0(\hat X_{t_{i}}) + \sum_{j=1}^k W^j_{t_i,t} \sigma_j(\hat X_{t_{i}})  \\
& \qquad + \frac{1}{2} \sum_{j,l=1}^k  W_{t_{i},t}^j W_{t_i,t}^{l} \bar{\nabla}_{\sigma_j} \sigma_{l}(\hat X_{t_{i}})  + \sum_{1\leq j<l\leq k} A^{j,l}_{t_i,t} [\sigma_j, \sigma_l](\hat X_{t_{i}}).
\end{align*}
\end{remark}

\subsection{Heisenberg group} \label{sec:Heis}
Let us consider the space $M=\mathbb{R}^{2k+1}$ with coordinates $q =(x,y,z)$ for $x,y \in \mathbb{R}^k$ and $z \in \mathbb{R}$. We equip $M$ with a multiplication rule such that if $q = (x,y,z)$ and $\tilde q = (\tilde x, \tilde y, \tilde z)$, then
\begin{align*}
    \ell_{\tilde q}(q) =\tilde q \cdot q
    & = \left(\tilde x+x, \tilde y+y, \tilde z+ z+\frac{1}{2} \left(\langle \tilde x, y \rangle_{\mathbb{R}^k} - \langle x, \tilde y\rangle_{\mathbb{R}^k}\right) \right).
\end{align*} 
Here, $\ell_{\tilde q}$ denotes the left translation with respect to $\tilde q$, i.e., multiplication by $\tilde q$ on the left. Observe that this multiplication is not abelian. With respect to this group structure, we have the identity element $0 =(0,0,0)$ and inverses given by $(x,y,z)^{-1} = (-x,-y,-z)$. We further define the sub-Riemanian structure $(E,g)$ on $M$ where $E$ is the rank $2k$ subbundle spanned by the vector fields
$$\sigma_j = \partial_{x^j} - \frac{y^j}{2} \partial_z, \qquad  \tau_j = \partial_{y^j} + \frac{x^j}{2} \partial_z, \qquad j\in\{1, \dots, k\},$$
and the sub-Riemannian metric $g$ is defined uniquely by requiring $(\sigma_1,\tau_1,\dots,\sigma_k,\tau_k)$ to be an orthonormal frame for $E$. The subbundle $E$ is bracket-generating because $[\sigma_i, \tau_j] = \delta_{i,j} \partial_z$.
We additionally observe that the above vector fields are \emph{left-invariant}, meaning that, for any smooth function $\varphi\colon M\to\R$ and for $q\in M$, we have
$$\sigma_j (\varphi \circ \ell_q) = (\sigma_j \varphi) \circ \ell_q,\qquad \tau_j(\varphi \circ \ell_q) = (\tau_j \varphi) \circ \ell_q.$$
As a consequence, if $d_g$ denotes the sub-Riemannian distance on $(M,E,g)$ and we write $f(q) := d_g(0,q)$ then
$$d_g(\tilde q, q) = d_g(0, \tilde q^{-1} \cdot  q) = f( \tilde q^{-1} \cdot q)  .$$
If we choose $\db\mu$ to be the usual Lebesgue measure on $\mathbb{R}^{2k+1}$ then the corresponding sub-Laplacian $\Delta=\Delta_{E,g,\db\mu}$ on $M=\R^{2k+1}$ is given by
$$\Delta = \sum_{j=1}^k \sigma_j^2 + \sum_{j=1}^k \tau_j^2.$$
This sub-Laplacian $\Delta$ is also left-invariant, i.e., for a smooth function $\varphi\colon M\to\R$ and $q\in M$, we have $\Delta (\varphi \circ \ell_q) = (\Delta \varphi) \circ \ell_q$. We now consider the stochastic process $(X_t)_{t\geq 0} = (X_t^{0})_{t\geq 0}$ on $M$ with infinitesimal generator $\frac{1}{2}\Delta$ and starting from $0\in M$. Alternatively, $(X_t)_{t\geq 0}$ can be characterised as the unique strong solution to the Stratonovich stochastic differential equation
\begin{equation} \label{SDEHeis} \dd X_t = \sum_{j=1}^k \sigma_j(X_t) \circ  \db W^j_t + \sum_{j=1}^k \tau_j(X_t) \circ \db W^{k+j}_t,\end{equation}
subject to $X_0=0$ and where $(W_t^1, \dots, W^{2k}_t)_{t\geq 0}$ is a standard Brownian motion on~$\R^{2k}$. The stochastic differential equation~\eqref{SDEHeis} subject to $X_0=0$ has the explicit solution existing for all times $t\geq 0$ and defined by
\begin{align*}
X_t & = \left(W_t^1, \dots, W^k_t, W^{k+1}_t,\dots,W_t^{2k}, \frac{1}{2} \sum_{j=1}^k \int_0^t (W^{j}_s \circ \db W_s^{k+j} - W^{k+j}_s \circ \db W_s^{j})  \right) \\
& = \left(W_t^{1}, \dots, W_t^{2k}, \sum_{j=1}^k A^{j,k+j}_{0,t} \right).
\end{align*}
If we take another initial point $q \in M$, then the solution $(X_t^{q})_{t\geq 0}$ to \eqref{SDEHeis} subject to $X_0^q=q$ is given by $X_t^{q} = \ell_{q}(X_t)$ for $t\geq 0$. Furthermore, we have
$$X_{t_i +t} = X_{t_i} \cdot X_{t_i,t_i+t} \quad\text{with}\quad X_{t_i,t_i+t} = \left(W_{t_i,t_i+t}^1,\dots, W_{t_i,t_i+t}^{2k}, \sum_{j=1}^k A_{t_i,t_i+t}^{j,k+j}\right).$$

From the explicit expression above, we see that the stochastic process $(X_t)_{t\geq 0}$ and the associated density $p_t(0,q) = p_t(q)$ can be simulated using approximations of the Lévy area described in Section~\ref{sec:TaylorApproximations}. We further note that since the sub-Laplacian $\Delta$ is left-invariant, we have $p_t(\tilde q, q) = (p_t \circ \ell_{\tilde q^{-1}})(q)$.

Let us now look at an explicit formula for $p_t$. By adapting the expression derived in~\cite{gaveau1977principe} to our normalisation for the vector fields used to define the sub-Riemannian structure $(E,g)$, we obtain
 $$p_t(q) = \frac{4}{(2\pi t)^{k+1}} \int_{-\infty}^{\infty} \left( \frac{2\lambda}{\sinh 2\lambda} \right)^k  \exp\left(\frac{4\im\lambda z}{t}-2\lambda \coth(2\lambda) \frac{\|x\|_{\R^k}^2 +\|y\|_{\R^k}^2}{2t}\right)\dd\lambda.$$
As further shown in~\cite{gaveau1977principe}, we then have the approximation, as $t\downarrow 0$,
$$\log p_t(q) \approx -\frac{d_g(0,q)^2}{2t} = - \frac{f(q)^2}{2t},$$
and consequently
$$\log p_t(\tilde q, q) = \log p_t(\tilde q^{-1} \cdot q) \approx - \frac{f(\tilde q^{-1} \cdot q)^2}{2t} = -\frac{d_g(\tilde q,q)^2}{2t}.$$
Unfortunately, computing the distance $f$ explicitly is very expensive. However, it is well known, see e.g.~\cite{ABB}, that
$$f(x,y,0) = \sqrt{\|x\|_{\mathbb{R}^k}^2 + \|y\|_{\mathbb{R}^k}^2} \qquad\text{and}\qquad f(0,0,z) = 2 \sqrt{\pi |z|}.$$
Define $\hat f\colon M\to[0,\infty)$ by
$$\hat f(x,y,z)^2 = f(x,y,0)^2 + f(0,0,z)^2 = \| x\|_{\mathbb{R}^k}^2 + \| y\|_{\mathbb{R}^k}^2 +4 \pi |z|.$$
Following the proof of \cite[Example~4.1]{grong2022geometric}, we also have $f(x,y,0) \leq f(x,y,z)$ and $f(0,0,z) \leq 2f(x,y,z)$. Hence, we obtain that
\begin{align*}
\frac{1}{8} \hat f(x,y,z)^2 & \leq \frac{1}{4}\max \{ f(x,y,0)^2, f(0,0,z)^2\} \\
& \leq f(x,y,z)^2  \leq ( f(x,y,0) + f(0,0,z))^2 \leq 2\hat f(x,y,z)^2.
\end{align*}
If we further consider $\hat d\colon M\times M\to\R$ given by $\hat d(\tilde q, q) = \hat f(\tilde q^{-1} \cdot q)$, then $\hat d$ and $d_g$ are Lipschitz equivalent left-invariant metrics agreeing on the distance from $0$ to points on the $z$-axis and to points in the $(x,y)$-plane. We therefore introduce the following score approximation
$$\hat S_t(\tilde q, q) := - \frac{1}{2t} \nabla^{q,E} \hat d(\tilde q, q)^2 = -\frac{1}{2t} \nabla^{q,E} \hat f(\tilde q^{-1} \cdot q)^2.$$
We further note that
\begin{align*}
& \nabla^{E} \left(\hat f^2\circ \ell_{\tilde q^{-1}}\right)(q) = \nabla^{q,E} \hat f\left(\tilde q^{-1} \cdot q  \right)^2 \\
& = \sum_{j=1}^k \left( 2 (x^j-\tilde x^j)- 2\pi(y^j - \tilde y^j)\sgn\left(z - \tilde z - \frac{1}{2} (\langle \tilde x,  y \rangle_{\mathbb{R}^k} - \langle x, \tilde y\rangle_{\mathbb{R}^k})\right)\right) \sigma_j( q)\\
& \quad + \sum_{j=1}^k \left(2  (y^j - \tilde y^j) + 2\pi (x^j - \tilde x^j)\sgn\left(z - \tilde z - \frac{1}{2} (\langle \tilde x,  y \rangle_{\mathbb{R}^k} - \langle x, \tilde y\rangle_{\mathbb{R}^k})\right)\right) \tau_j(q).
\end{align*}
It follows that if writing $\hat S_t(\tilde q, q) = \sum_{j=1}^k \hat S_t^j(\tilde q, q) \sigma_j(q) + \sum_{j=1}^k \hat S_t^{k+j}(\tilde q, q) \tau_j(q)$
then, for $l\in\{1,\dots,2k\}$,
$$\hat S_t^{l}(\tilde q, q) = \hat S_t^{l}(0,\tilde q^{-1} \cdot q) =: \hat S_t^{l}(\tilde q^{-1} \cdot q),$$
which are given by, for $j\in\{1, \dots, k\}$,
\begin{displaymath}
    -t \hat S_t^j(q) =  x^j- \pi y^j\sgn\left(z \right)\qquad\text{and}\qquad
    -t \hat S_t^{k+j}(q) = y^j  + \pi x^j\sgn\left(z\right).
\end{displaymath}

We now continue with the following procedure. Suppose we have a discretisation  $0 = t_0 < t_1 < \cdots < t_n = T$ with $t_{i+1} - t_i =\delta = \frac{T}{n}$. For each $t\in(0, \delta]$, we define
$$\hat X_{t_i+t} = \hat X_{t_i} \cdot \hat X_{t_i,t_i+t},$$
with
$$\hat X_{t_i,t_i+t} = \left(W_{t_i,t_i+t}^{1}, \dots, W_{t_i,t_i+t}^{2k},  \sum_{j=1}^k \hat A^{j,k+j}_{t_i,t_i+t} \right).$$
Letting $\hat \pp_{s,t}$ denote the probability measure on $M \times M$ associated with $(\hat X_s,\hat X_t)$, we then define the approximation $\hat e_{s,t}(\theta)$ to $e_{s,t}(\theta)$ by
$$
\hat e_{s,t}(\theta) = \int_{M \times M} \langle S^\theta_t(y), S^\theta_t(y) - 2\hat S_{t-s}(z,y)\rangle^2_{g(y)} \, \db \pp_{s,t}(z,y).$$
We write $S^\theta_t = \sum_{j=1}^k S_t^{\theta,j} \sigma_j + \sum_{j=1}^k S_t^{\theta,k+j} \tau_j$. In terms of $\hat \bfS_t = (\hat S^1_t, \dots, \hat S^{2k}_t)$ as well as $\bfS_t^\theta = (S^{\theta,1}_t, \dots, S^{\theta,2k}_t)$, we obtain
\begin{align*}
\hat e_{t_i,t_i+t}(\theta) &= \E\left[\left\langle \bfS^\theta_{t_i+t}(\hat X_{t_i+t}), \bfS^\theta_{t_i+t}(\hat X_{t_i+t}) - 2\hat \bfS_t(\hat X_{t_i,t_i+t} ) \right\rangle_{\mathbb{R}^{2k}} \right] \\
&= \E\left[\left\langle \bfS^\theta_{t_i+t}(\ell_{\hat X_{t_i}}\hat X_{t_i,t_i+t}), \bfS^\theta_{t_i+t}(\ell_{\hat X_{t_i}}\hat X_{t_i,t_i+t}) - 2\hat \bfS_t(\hat X_{t_i,t_i+t} ) \right\rangle_{\mathbb{R}^{2k}} \right]. 
\end{align*}
We remark that
$$-\hat \bfS_t(q) = \frac{1}{t} \begin{pmatrix} x \\ y \end{pmatrix} + \frac{\pi}{t} \sgn(z) \begin{pmatrix} -y \\ x \end{pmatrix}  =: - \frac{1}{t} \hat \bfS(q).$$

Using these observation in summary, we describe how $\hat \scrE_{0,T}(\theta)$ is presented with a discretisation of $[0,T]$ for $\delta = \frac{T}{n}$. We generate $K \cdot n \cdot 2k$ random samples $\Delta_i W^{(l),j}$ from the distribution $N(0,\delta)$, and further $K \cdot K_2 \cdot n \cdot 2k$ numbers $\Delta_i c_{j,m}^{(l)}$ sampled from $N(0, \frac{\delta}{2m+1})$. Here $1 \leq i \leq n$, $1\leq j \leq 2k$, $1 \leq l\leq K$ and $1 \leq m \leq K_2$. We then determine
\begin{align*}
    \Delta_i A^{(l), j,k+j} &= \frac{1}{2}\left(\Delta_i c_{k+j,1}^{(l)} \Delta_i W^{(l),j} - \Delta_i c_{j,1}^{(l)} \Delta_i W^{(l),k+j} \right) \\
    & \qquad +\frac{1}{2}\sum_{m=1}^{K_2-1} \left( \Delta_i c_{j,m}^{(l)} \Delta_i c_{k+j,m+1}^{(l)} - \Delta_i c_{j,m+1}^{(l)} \Delta_i c_{k+j,m}^{(l)} \right)
\end{align*}
and set
$$\Delta_i A^{(l)} = \sum_{j=1}^k \Delta_i A^{(l),j,k+j}, \qquad \Delta_i \hat X^{(l)} = (\Delta_i W^{(l),1},\dots, \Delta_i W^{(l),2k}, \Delta_i A^{(l)}).$$
We further define $\hat X_0^{(l)} =x_0$ and iteratively
$$\hat X_{i+1}^{(l)} = \hat X_{i}^{(l)} \cdot \Delta_{i+1} \hat X^{(l)}.$$
We compute the loss function by
\begin{align} \label{eq:denoising_loss_Heisenberg}
\hat \scrE_{0,T}(\theta)
& \approx \frac{1}{K} \sum_{l=1}^K \sum_{i=1}^n \left\langle \bfS^\theta_{i \delta}(\hat X_{i}^{(l)}), \delta \, \bfS^\theta_{i \delta}(\hat X_{i}^{(l)}) - 2\hat \bfS(\Delta_i \hat X^{(l)} ) \right\rangle_{\mathbb{R}^{2k}} \\ \nonumber
& = \frac{\delta}{K} \sum_{l=1}^K \sum_{i=1}^n \left\| \bfS^\theta_{i \delta}(\hat X_{i}^{(l)}) - \frac{1}{\delta} \hat \bfS(\Delta_i \hat X^{(l)} ) \right\|_{\mathbb{R}^{2k}}^2+ C
\end{align}
and can now apply Algorithm~\ref{algo:Divergence} with the loss $\scrE_{0,T}(\theta)$ replaced by $\hat\scrE_{0,T}(\theta)$.

\subsection{General case and adapted coordinates} \label{sec:AdaptedCoordinates}
Despite us using very particular properties of the Heisenberg group in Section~\ref{sec:Heis},
it is possible to perform local approximations for a general sub-Riemannian manifold in a similar manner.

Let $(M,E,g)$ be a sub-Riemannian manifold where the manifold $M$ has dimension $d$. Let $p_t$ be the Dirichlet heat kernel for an operator $\frac{1}{2} L =\frac{1}{2} \Delta + Z$.
We will make use of small-time asymptotics, see~\cite{leandre1987majoration,leandre1987minoration}, saying that, for any $x_1, x_2 \in M$, the transition density $p_t$ satisfies the limit
\begin{equation}\label{eq:sRvaradhan}
    \lim_{t\to 0} 2t \log p_{t}(x_1,x_2) = - d_g(x_1,x_2)^2.
\end{equation}
In order to use this result, we need an approximation of the distance $d_g$.

For the remaining results in this section, we follow Montgomery~\cite[Chapter~2.4]{montgomery2002tour}. We consider \emph{adapted coordinates}, which are also called \emph{privileged coordinates} in~\cite{bellaiche1996tangent}. For a given point $x_0 \in M$, we define a collection of integers $k_1,k_2,\dots,k_s$ which satisfy $k= k_1 \leq  k_2 \leq \cdots \leq k_s=d$ such that $k_1$ is the rank of $E$, $k_2$ is the rank of the space spanned by vector fields in $E$ along with their first order Lie brackets at $x_0$, $k_3$ is the rank of the space which also includes their second order Lie brackets at $x_0$, and so on. Due to the bracket-generating assumption on $E$, there exists some minimal $s\in\N$ such that $k_s=d$, which is \emph{the step} of $E$. The vector $(k_1, \dots, k_s)$ is then called \emph{the growth vector} of $E$ at $x_0$.
For instance, in the $(2k+1)$-dimensional Heisenberg group, the growth vector is $(2k,2k+1)$ at all points. With the convention that $k_0=0$, we further define \emph{weights} $\nu_1, \dots, \nu_d$ such that $\nu_i$ for $i\in\{1,\dots, d\}$ is the maximal number satisfying
$$k_{\nu_i-1} \leq i.$$
In the Heisenberg group, we have $(\nu_1,\dots,\nu_{2k},\nu_{2k+1}) =(1,\dots, 1,2)$.

A coordinate system $(y^1,\dots,y^d)$ centered at $x_0$ is then called \emph{adapted at $x_0$} if, for any $l \leq \nu_i-1$ and for any selection of vector fields $V_1,\dots, V_l$ in $E$,
$$\left(V_1\dots V_l y^i\right)(x_0) = 0.$$
Note that the condition for $\nu_i =1$ is already satisfied for all coordinates because $y(x_0) = 0$ by assumption. Relative to an adapted coordinate system, we define \emph{the box metric} $d_{box}$ by, for $x_1$ and $x_2$ in the domain of the coordinate system $y$,
$$d_{box}(x_1, x_2) =\max_{1\leq i \leq d} \left\{ |y^i(x_1) - y^i(x_2)|^{1/\nu_i}\right\}.$$
\emph{The ball box theorem} says, for all $x$ in the domain of $y$ and for constants $c$ and $C$,
$$c d_{box}(x_0,x) \leq d_g(x_0,x) \leq C d_{box}(x_0,x).$$
In order to have a smoother metric, we instead consider $\hat d$ defined by
$$\hat d(x_1,x_2)^2 = \sum_{i=1}^d |y^i(x_1) - y^i(x_2)|^{2/\nu_i}.$$
The previous inequalities imply that there are positive constants $\tilde c$ and $\tilde C$ such that
\begin{equation}\label{eq:dhatvsdg}
    \tilde c \hat d(x_0,x)^2 \leq d_g(x_0,x)^2 \leq \tilde C \hat d(x_0,x)^2.
\end{equation}
Motivated by \eqref{eq:sRvaradhan} and \eqref{eq:dhatvsdg}, we introduce the approximation
$$\log p_t(x_0,x) \approx  -\frac{\hat d(x_0,x)^2}{2t},$$
leading to our final approximated score
$$\hat S_t(x_0,x) := -\frac{1}{2t}\nabla^{x,E} \hat d(x_0,x)^2.$$

We conclude that if $(\sigma_1, \dots, \sigma_k)$ is an orthonormal frame for $E$ with respect to $g$ and we write $\sigma_j = \sum_{i=1}^d \sigma^i_j \partial_{y^i}$ for $j\in\{1,\dots,k\}$, then our approximation of the score is given by
\begin{equation}\label{eq:approximatescore}
\hat S_t(x_0,x) = - \frac{1}{2t} \nabla^{x,E} \hat d(x_0,x)^2 = - \frac{1}{t} \sum_{j=1}^k \left( \sum_{i=1}^d \sigma_j^i(x) \frac{ y^i |y^i|^{2/\nu_i-2}}{\nu_i} \right) \sigma_j(x).
\end{equation}

\begin{remark}
Note that, for any selection of strictly positive constants $c^1,\dots, c^d$, the metric $\hat d^c$ defined by
$$\hat d^c(x_1,x_2)^2 = \sum_{i=1}^d c^i |y^i(x_1) - y^i(x_2)|^{2/\nu_i}$$
will also satisfy \eqref{eq:dhatvsdg}. Hence, for each case, we can scale this metric to make it match the true sub-Riemannian metric close to $x_0$ for each coordinate, as we did for the Heisenberg group in Section~\ref{sec:Heis}. A recommended method for finding a good choice of $c^1, \dots, c^d$ is to compute explicit lengths along each coordinate axis according to the sub-Riemannian metric of the nilpotent homogeneous approximation at $x_0$, as described in \cite[Section~5.3]{bellaiche1996tangent}.
\end{remark}

\subsection{Simulation of a step in the general geometry} \label{sec:SimStep}
For simulating the stochastic process $(X_t)_{t\in[0,T]}$ in the general sub-Riemannian setting, we use the stochastic Taylor expansion discussed in Section~\ref{sec:TaylorApproximations} and the adapted coordinates reviewed in Section~\ref{sec:AdaptedCoordinates}. For the sake of simplicity, we reduce our considerations to the case where the bracket-generating distribution $E$ has step 2, that is, where the tangent bundle $TM$ is spanned by $E$ and its first order Lie brackets.

As before, let $(\sigma_1, \dots, \sigma_k)$ be an orthonormal frame for $E$, let $(W_t)_{t\geq 0}$ be a standard Brownian motion on $\R^k$ and assume that $(X_t)_{t\in[0,T]}$ is the unique strong solution to the Stratonovich stochastic differential equation
$$\db X_t = \sum_{j=1}^k \sigma_j(X_t) \circ \db W_t^j + \sigma_0(X_t) \dd t,\qquad X_0=x_0.$$

Additionally consider vector fields $\sigma_{k+1}, \dots, \sigma_{d}$ on $M$ such that $(\sigma_1, \dots, \sigma_d)$ is a frame for the full tangent bundle $TM$.
Let $(x^1, \dots, x^d)$ be any choice of local coordinate system for $M$ around the point $x_0 = (x^1_0, \dots, x^d_{0})\in M$. We introduce a matrix-valued function $\sigma =(\sigma^i_{l})$ such that $\sigma_l = \sum_{i=1}^d \sigma^i_{l} \partial_{x^i}$ for $l\in\{1,\dots,d\}$, and we denote its inverse by $\sigma^{-1} = (\sigma_{-l}^i)$. We then define a new local coordinate system $(y^1, \dots, y^d)$ for $M$ centred at $x_0$ by setting
\begin{equation} \label{AdaptedStep2} y^i = \sum_{l=1}^d \sigma^{i}_{-l}(x_0) (x^l- x^l_{0}).\end{equation}
This yields an adapted coordinate system at $x_0$.
Letting $(Y_t)_{t\in[0,T]}$ defined by $Y_t = y(X_t)$ be the process in this adapted coordinate system and taking the Taylor expansion of order 1, we obtain that, for $t\in[0,\delta]$ and $i\in\{1,\dots,k\}$ as well as $m\in\{k+1,\dots,d\}$,
$$Y_t^i \approx  W_t^i + \sum_{l=1}^d \sum_{j_1,j_2=1}^k \left(\int_0^t W_s^{j_1} \circ \db W_s^{j_2} \right) \sigma^{i}_{-l}(x_0) (\bar{\nabla}_{\sigma_{j_1}}\sigma_{j_2})^l(x_0) + \sum_{l=1}^d t \sigma^{i}_{-l}(x_0) \sigma^l_0(x_0) $$
as well as
$$Y_t^{m} \approx  \sum_{l=1}^d \sum_{j_1,j_2=1}^k \left(\int_0^t W_s^{j_1} \circ \db W_s^{j_2} \right) \sigma^{m}_{-l}(x_0) (\bar{\nabla}_{\sigma_{j_1}} \sigma_{j_2})^l(x_0) + \sum_{l=1}^d t \sigma^{m}_{-l}(x_0) \sigma^l_0(x_0).$$
Returning to our original coordinate system $(x^1,\dots,x^d)$ and identifying vectors in the basis $(\partial_{x^1}, \dots, \partial_{x^d})$ with points in the domain of $(x^1,\dots,x^d)$, we get the approximation
$$X_t - x_{0} \approx \sum_{j=1}^k W_t^j \sigma_{j}(x_0)  + \sum_{j_1,j_2=1}^k \left(\int_0^t W_s^{j_1} \circ \db W_s^{j_2} \right) (\bar{\nabla}_{\sigma_{j_1}}\sigma_{j_2})(x_0) + t \sigma_0(x_0).$$
We now have the following beneficial proposition.
\begin{proposition}\label{propn:approxpropn}
Suppose that $\spn\{ \sigma_j, [\sigma_{j}, \sigma_{l}] : j,l\in\{1,\dots, k\}\}$ equals the full tangent space at every point on $M$. For a given point $x_0\in M$,  we define a stochastic process $(\hat X_t)_{t\geq 0}$ by $\hat X_0 = x_0$ and
\begin{equation}\label{eq:approxpropn}
    \hat X_t - x_{0} = \sum_{j=1}^k \sigma_{j}(x_0) W_t^j + \sum_{j_1,j_2=1}^k  (\bar{\nabla}_{\sigma_{j_1}}\sigma_{j_2})(x_0) \int_0^t W_s^{j_1} \circ \db W_s^{j_2}+ \sigma_0(x_0)t.
\end{equation}
Then $\hat X_t$  for $t>0$ has a positive smooth density with respect to Lebesgue measure.
\end{proposition}
We remark that by using the identity $\int_0^t W_s^{j_1} \circ \db W_s^{j_2} = \frac{1}{2} W_t^{j_1} W_t^{j_2} + A^{j_1,j_2}_{0,t}$ we can rewrite the expression~\eqref{eq:approxpropn} as
\begin{align*}
\hat X_t - x_{0} & = \sum_{j=1}^k \sigma_{j}(x_0) W_t^j + \frac{1}{2} \sum_{j_1,j_2=1}^k  (\bar{\nabla}_{\sigma_{j_1}}\sigma_{j_2})(x_0) W_t^{j_1} W_t^{j_2} \\
& \qquad + \sum_{1\leq j_1<j_2\leq k}  [\sigma_{j_1},\sigma_{j_2}](x_0) A^{j_1,j_2}_{0,t} + \sigma_0(x_0)t.
\end{align*}

\begin{proof}[Proof of Proposition~\ref{propn:approxpropn}]
We consider the space $\R^k \times \R^d$ with coordinates $(w,x)$ and we introduce the vector fields, for $j\in\{1,\dots,k\}$,
\begin{align*}
\hat \sigma_j(w,x) & = \partial_{w_j} + \sum_{i=1}^d \sigma_{j}^i(x_0) \partial_{x_i} + \sum_{j_1=1}^k \sum_{i=1}^d w^{j_1}(\bar{\nabla}_{\sigma_{j_1}}\sigma_{j})^i(x_0) \partial_{x_i} \\
& = \partial_{w_j} +\sigma_j(x_0)  + \sum_{j_1=1}^k w^{j_1} (\bar{\nabla}_{\sigma_{j_1}}\sigma_{j})(x_0).
\end{align*}
We notice that, for $j,l\in\{1,\dots,k\}$ and at any point $(w,x)\in \R^k \times \R^d$,
$$[\hat \sigma_j, \hat \sigma_{l}](w,x) = \sum_{i=1}^d [ \sigma_j, \sigma_{l}]^i(x_0) \partial_{x^i}.$$

Let us now consider the set $O\subseteq \R^k \times \R^d$ of points that can be reached from $(0,x_0)$ by moving tangentially to the subbundle $\hat E = \spn\{ \hat \sigma_j:j\in\{1,\dots,k\}\}$, and further define the constant vector space $R = \spn\{ \sum_{i=1}^d [\sigma_{j_1}, \sigma_{j_2}]^i(x_0) \partial_{x^i}\}_{1\leq j_1 < j_2\leq k}$. By the Orbit theorem \cite{sussmann1973orbits}, the set $O$ is a manifold of dimension $k + \rank R$, whose tangent space is given by
$$T_{(w,x)}O = \hat E_{(w,x)} \oplus R_{(w,x)}, \qquad (w,x) \in O.$$ 
Let $\pi\colon \R^k \times \R^d \to \R^d$ be the projection on the second factor. By our assumption on the brackets, we have $\pi_* (\hat E_{(w,x)}\oplus R_{(w,x)}) = T_x \R^d$ for any $(w,x) \in \R^k \times \R^d$. It follows that $\pi|_O\colon O \to \R^d$ is a surjective submersion with each fiber $\pi^{-1}(x)$ being an embedded manifold.

We observe that $\pi_{*,(w,x)}$ maps $(\ker \pi_{*,(w,x)})^\perp$ bijectively to $T_x\R^d$. Hence, we can define a vector field $\hat \sigma_0$ on $\R^k \times \R^d$ determined uniquely by being in $(\ker \pi_*)^\perp$ and satisfying $\pi_{*,(w,x)} \hat \sigma_0(w,x) = \sigma_0(x_0)$, where we identify the vector $\sigma(x_0)$ and the corresponding constant vector field.

Define a stochastic process $(\Phi_t)_{t\geq 0}$ on $\R^k\times\R^d$ as the unique strong solution to
$$\db\Phi_t = \sum_{j=1}^k \hat \sigma_j(\Phi_t) \circ \db W_t^j + \hat \sigma_0(\Phi_t) \dd t,\qquad \Phi_0 =(0,x_0).$$
Since $\hat E|_O$ is bracket-generating as a subbundle of $TO$, we deduce from \cite{Hor67} and \cite{StVa72} that $\Phi_t$ for $t>0$ has a positive smooth density $\varphi_t$ with respect to Lebesgue measure $\db \mu^O$ on $O$.
Furthermore, by definition we have $\pi(\Phi_t) = \hat X_t$ for $t\geq 0$. Hence, the density
$\hat p_t$ of $\hat X_t$ given by
$\hat p_t(x) = \int_{\pi^{-1}(x)} \varphi_t(w,x) \dd \mu^O|_{\pi^{-1}(x)}$ is the average of $\varphi_t$ on every fiber, which is then also smooth and positive for $t>0$.
\end{proof}

\subsection{General summary}
To summarise the above results, assume that $(E,g)$ is a step 2 sub-Riemannian structure with an orthonormal frame $(\sigma_1, \dots, \sigma_k)$ for $E$, extended with further vector fields $\sigma_{k+1}, \dots, \sigma_{d}$ to a frame for the tangent bundle $TM$, which in local coordinates $(x^1, \dots, x^d)$ is represented as a matrix $(\sigma^i_j)$ with inverse $(\sigma_{-j}^i)$. Let us divide $[0,T]$ into $n$ intervals of length $\delta = \frac{T}{n}$ and set $t_i = i \delta$ for $i\in\{0,\dots,n\}$. We then define an approximation $(\hat X_t)_{t\in[0,T]}$ to $(X_t)_{t\in[0,T]}$ by $\hat X_0 = x_0$ as well as, for $t\in(0,\delta]$ and $i\in\{0,\dots, n-1\}$,
\begin{align*}
\hat X_{t_{i} +t}  & = \hat X_{t_{i}}+ \sum_{j=1}^k \sigma_{j}(\hat X_{t_{i}}) W_{t_i,t_i+t}^j + \frac{1}{2} \sum_{j_1,j_2=1}^k  (\bar{\nabla}_{\sigma_{j_1}}\sigma_{j_2})(\hat X_{t_{i}})  W_{t_i,t_i+t}^{j_1} W_{t_i,t_i+t}^{j_2} \\
& \qquad + \sum_{1 \leq j_1 <j_2 \leq k}  [\sigma_{j_1}, \sigma_{j_2}](\hat X_{t_{i}})  \hat A^{j_1,j_2}_{t_i,t_i+t} + \sigma_0(\hat X_{t_i}) t,
\end{align*}
where $\hat A^{j_1,j_2}_{t_i,t_i+t}$ is a truncation of \eqref{LevyArea} or \eqref{eq:polyapprox}, respectively. We apply this model to sample from the density $\hat p_t(x_0,\cdot)$ of $\hat X_t$.

Moreover, we need a corresponding loss function for the score. For this, we use the adapted coordinates defined in \eqref{AdaptedStep2} for the approximation \eqref{eq:approximatescore} of the score, that is, with
\begin{align*}
\hat f(x)^2 = \hat d(x_0, x)^2  &= \sum_{i=1}^k |y^i|^2 + \sum_{i=k+1}^d |y^i| \\
& = \sum_{i=1}^k \left( \sum_{l=1}^d \sigma_{-l}^i(x_0) (x^l -x_0^l)\right)^2 + \sum_{i=k+1}^d\left|\sum_{l=1}^d \sigma_{-l}^i(x_0) (x^l -x_0^l) \right|,
\end{align*}
we take $\hat S_{t}(x_0,x) = - \frac{1}{2t} \nabla^E \hat f(x)^2$.
Writing
\begin{displaymath}
    \hat S_t(x_0,x)=- \frac{1}{2t} \sum_{j=1}^k \sigma_j(\hat f^2)(x) \sigma_j(x)
    =\sum_{j=1}^k \hat S_t^j(x_0,x) \sigma_j(x),
\end{displaymath}
we then have
\begin{align*}
\hat S_t^j(x_0,x) 
& = - \frac{1}{t} \sum_{i=1}^k \sum_{l_2=1}^d \sigma^{l_2}_j(x) \sigma_{-l_2}^i(x_0) \left( \sum_{l=1}^d \sigma_{-l}^i(x_0) (x^l -x_0^l)\right) \\
& \qquad - \frac{1}{2t} \sum_{i=k+1}^d \sum_{l_2=1}^d \sigma^{l_2}_j(x) \sigma_{-l_2}^i(x_0) \sgn\left(\sum_{l=1}^d \sigma_{-l}^i(x_0) (x^l -x_0^l) \right).
\end{align*}
We remark that the expression above differs from \eqref{eq:approximatescore} in that there the coordinate expressions $\sigma_j^i$ for the vector fields are defined relative to the adapted coordinate system $(y^1,\dots,y^d)$, whereas here these are relative to a given coordinate system $(x^1,\dots, x^d)$. A particular advantage of this formulation is that we do not need to change coordinate system if we change $x_0$.

We now define the loss function $\hat \scrE_{0,T}(\theta)$ for the neural network $S^\theta =\sum_{j=1}^k S^{\theta,j} \sigma_j$ using $\bfS^\theta = (S^{\theta,1}, \dots, S^{\theta,k})$  and $\hat{\bfS} = (\hat S^1, \dots, 
\hat S^k)$ by
\begin{align*}
    \hat \scrE_{0,T}(\theta) & = \int_0^T \E^{x_0}\left[\left\langle \bfS_t^\theta(\hat X_t), \bfS_t^\theta(\hat X_t)-2\hat{\bfS}_{t}(x_0,\hat X_t) \right\rangle_{\mathbb{R}^k}\right]\dd t \\
    & = \int_0^T \E^{x_0} \left[ \left\| \bfS^\theta_t(\hat X_t) - \hat \bfS_t(x_0, \hat X_t) \right\|_{\mathbb{R}^k}^2 \right] \dd t +C.
\end{align*}

\section{Experiments}
\label{sec:experiments}
We now exemplify the bridge simulation code for the Heisenberg group whose sub-Riemannian structure is described in Section~\ref{sec:Heis}. We use an Euler--Heun to simulate the unconditioned Brownian motion, and we employ this to train score approximators using the denoising loss \eqref{eq:denoising_loss_Heisenberg} and the divergence loss \eqref{eq:loss_samples}. Code for representing the Heisenberg geometry, geometric computations, and training of the score approximations is available in the Jax Geometry library\footnote{\url{https://github.com/computationalevolutionarymorphometry/jaxgeometry/}}.

The neural net used is a 3-layer dense network with 15 nodes in each layer and exponential linear units (ELU) activation functions. For the loss functions depending on the divergence, it is important that the activation functions are differentiable. In contrast, with the denoising loss, the ELU can be replaced with e.g. rectified linear units (ReLU).
The nets are trained for 2500 epochs with batch size of 64 sample paths and 8 batches per epoch.

In Figure~\ref{fig:bridge_sample}, we sample a single sample path (blue) starting at $x_0=(0.5,0,0.8)$ and conditioned on hitting $x_T=(0,0,0)$ at $T=0.1$. For comparison, the minimising geodesic (red) between the two points $x_0$ and $x_T$ is included in the figure. We additionally plot the norm of the $(x,y)$-component and the $z$-component, respectively, as a function of $t\in[0,0.1]$, particularly to illustrate the progression of the $z$ coordinate from its starting value towards $0$.
\begin{figure}[ht]
\centering
\includegraphics[width=.48\linewidth,trim=100 25 100 100,clip]{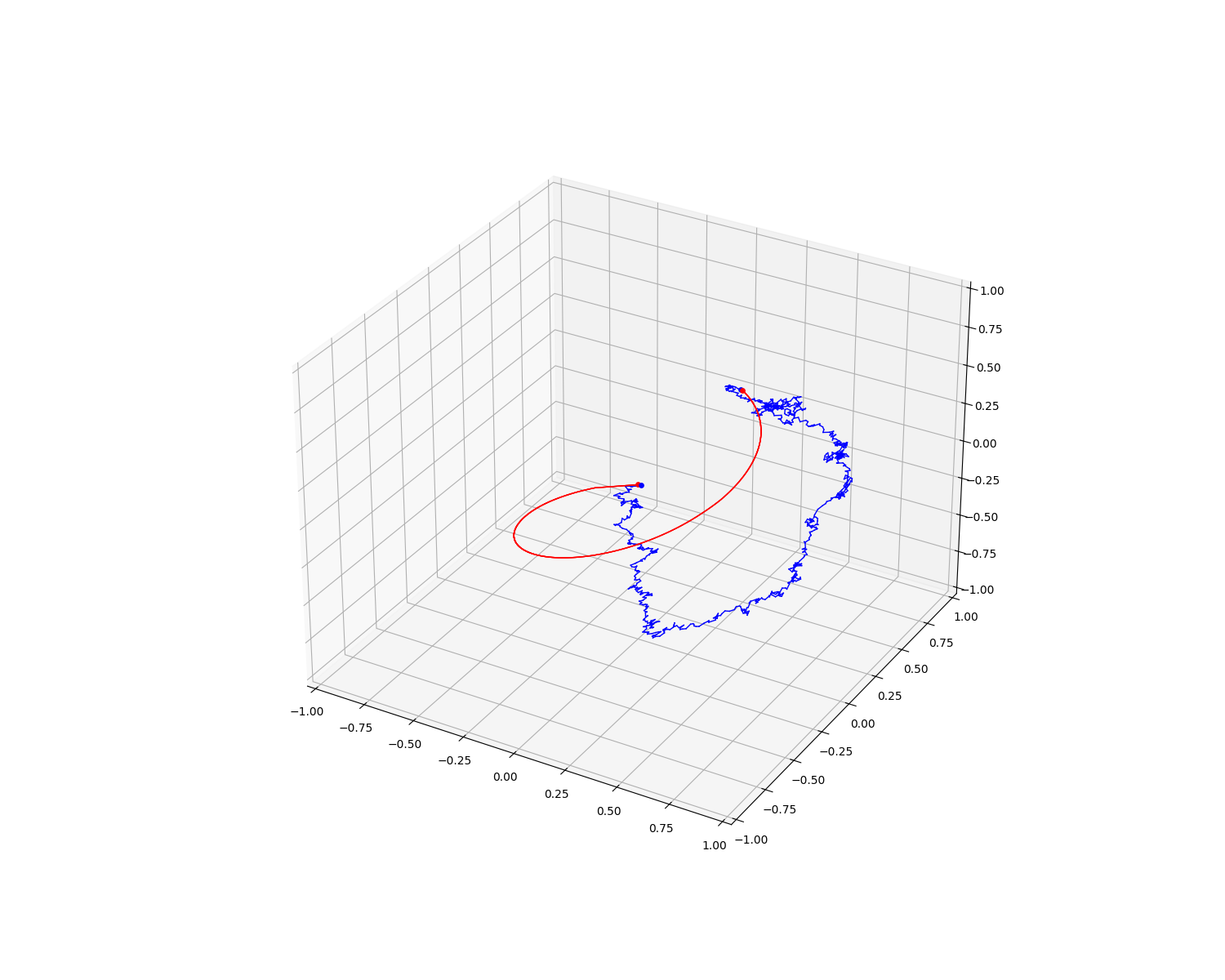}
\hfill 
\includegraphics[width=.48\linewidth,trim=100 25 100 100,clip]{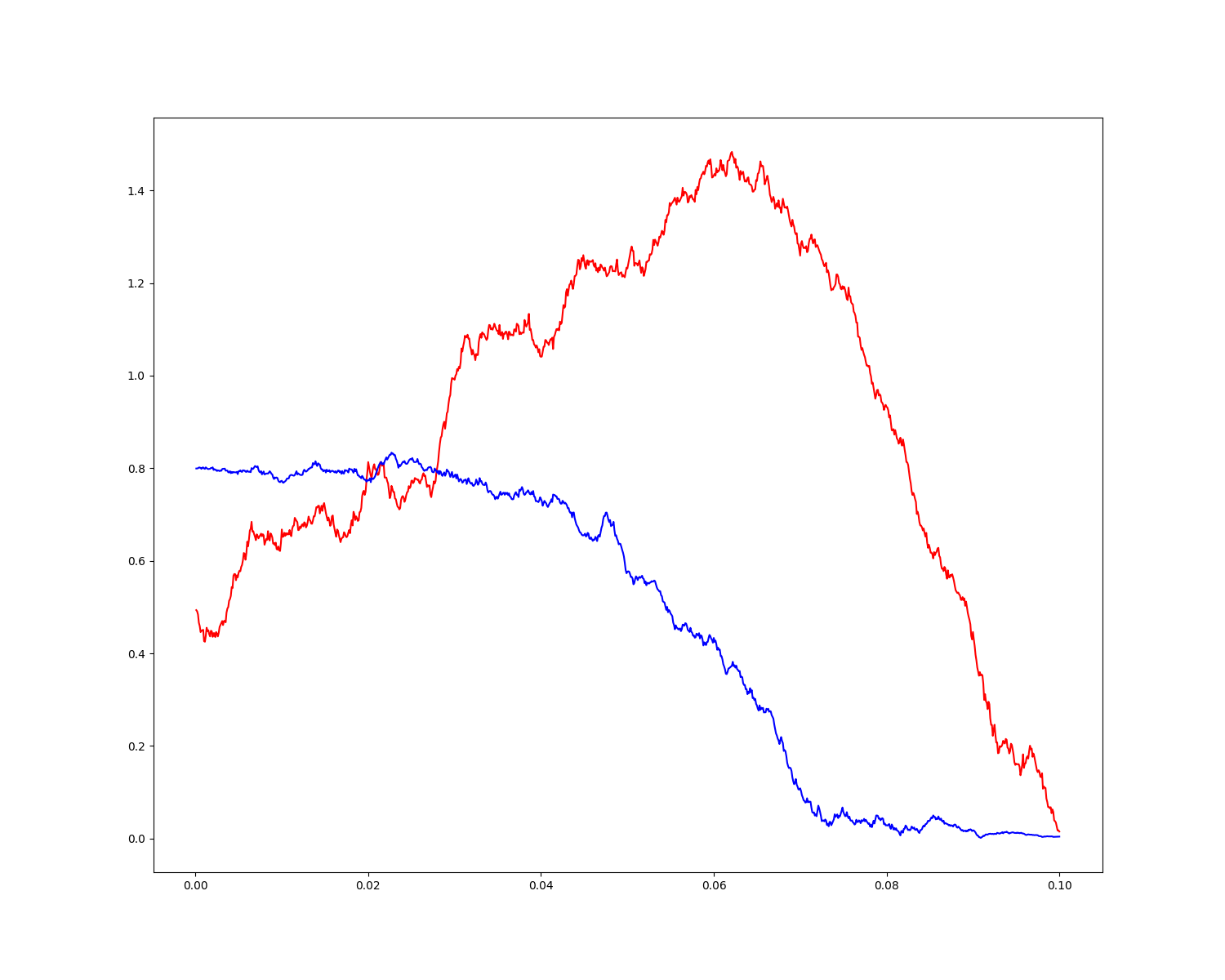}
\caption{Left: Sample path from bridge starting at $(0.5,0,0.8)$ conditioned on hitting $(0,0,0)$ at $T=0.1$ (blue curve). Geodesic between the same points corresponding to the limit $T\to 0$ (red curve). Right: Norm of $(x,y)$-component (red) and $z$-component (blue) for the sample path as a function of $t\in[0,0.1]$.}
\label{fig:bridge_sample}
\end{figure}

Since we know from short-time asymptotics established in~\cite{BN22} for diffusion processes on sub-Riemannian manifolds  that sample paths should concentrate around a geodesic if we reduce the conditioning time towards $0$, we plot in Figure~\ref{fig:bridge_stats} median curves and quartiles with conditioning for $T=0.1$, $T=0.2$, $T=0.5$ and $T=1$, respectively. The concentration effect is clearly visible.
\begin{figure}[ht]
\centering
\includegraphics[width=.45\linewidth,trim=100 25 100 100,clip]{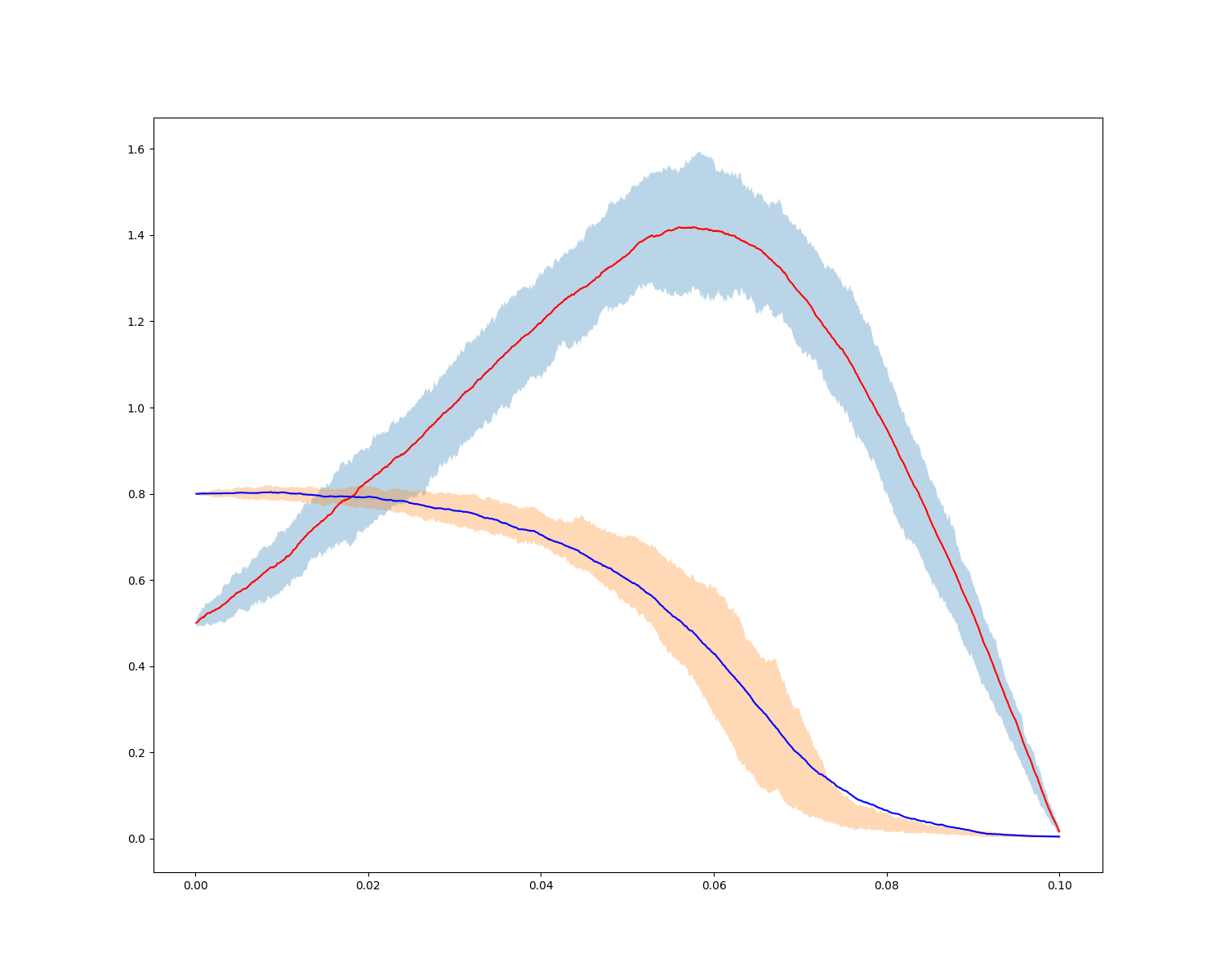}
\includegraphics[width=.45\linewidth,trim=100 25 100 100,clip]{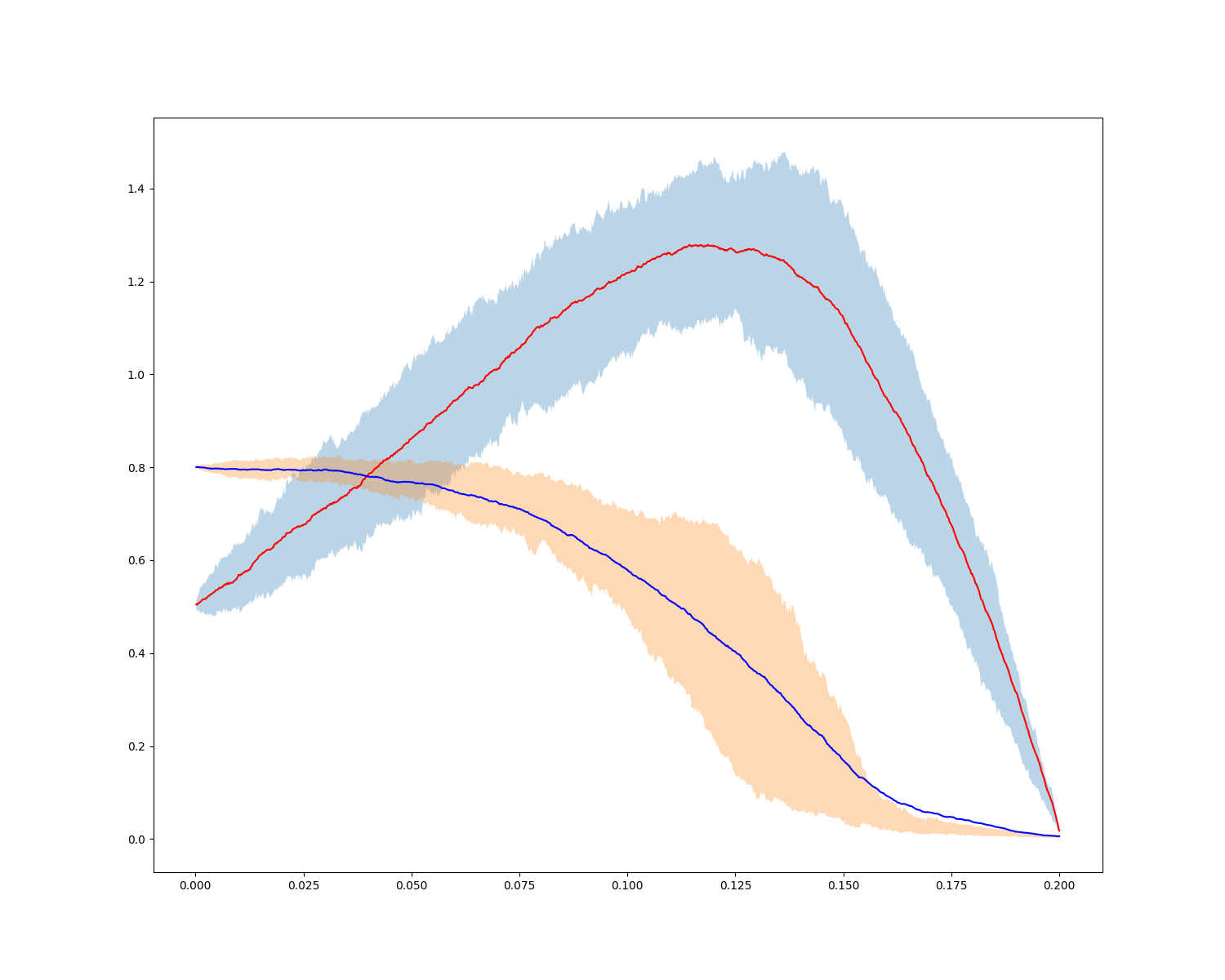}
\includegraphics[width=.45\linewidth,trim=100 25 100 100,clip]{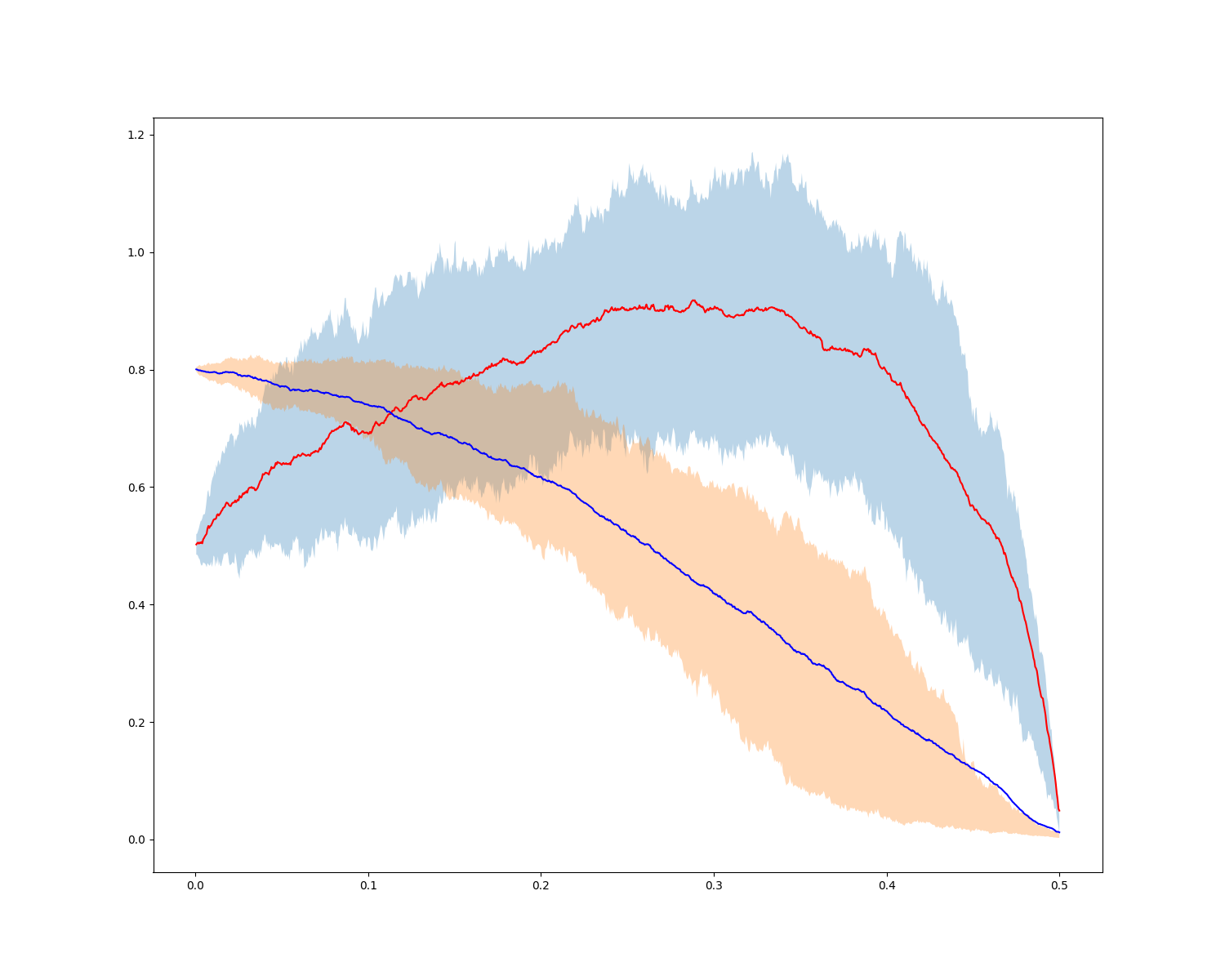}
\includegraphics[width=.45\linewidth,trim=100 25 100 100,clip]{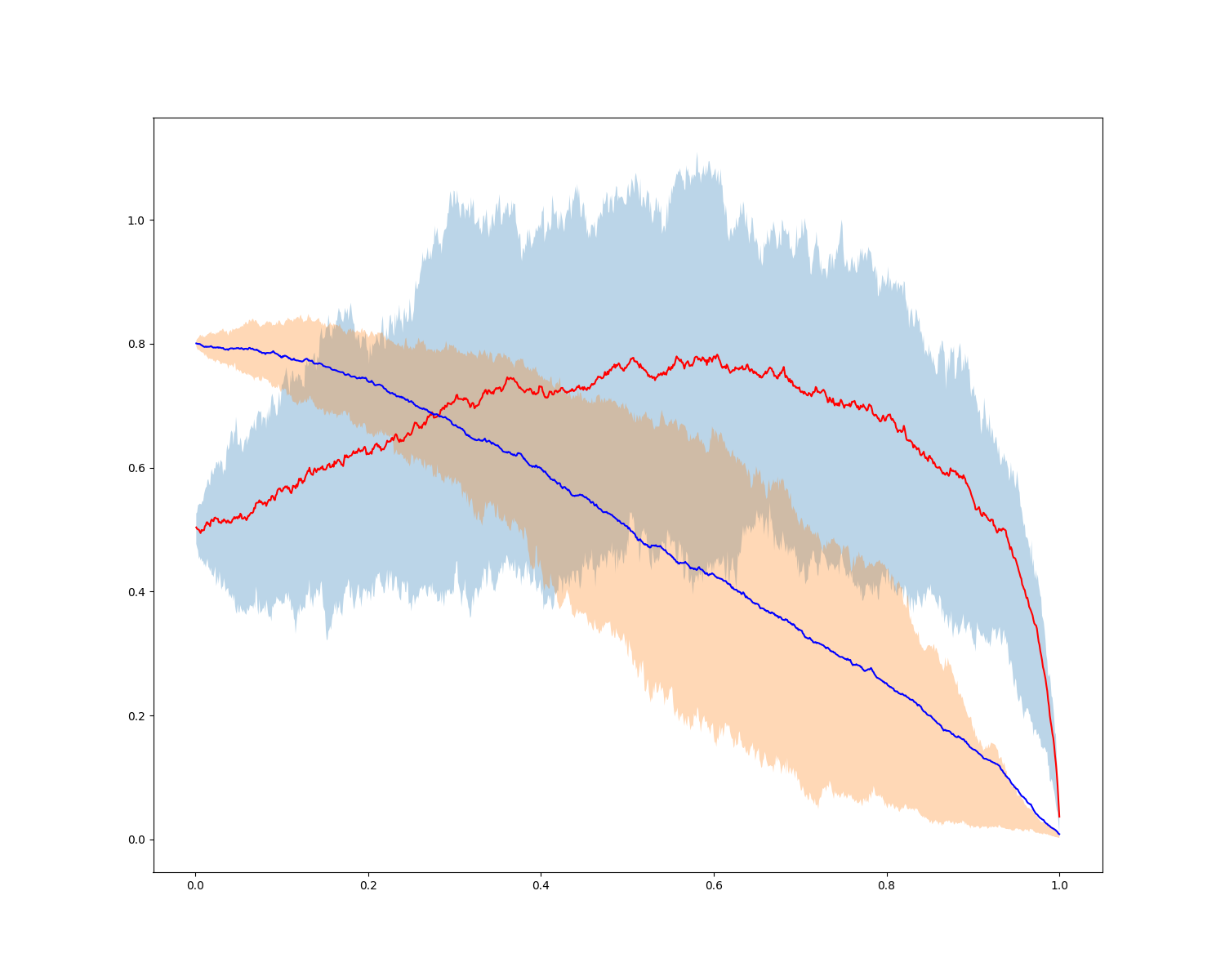}
\caption{Mean $(x,y)$-component and $z$-component norms over 100 samples with quartiles for $T=0.1$, $T=0.2$, $T=0.5$ and $T=1$, respectively.}
\label{fig:bridge_stats}
\end{figure}

Finally, Figure~\ref{fig:s1} shows plots of the estimated score fields in the unit cube in $\R^3$ as well as evaluated only on the $(x,y)$-slice of the unit cube satisfying $z=0.2$. The figure shows estimated score fields with the denoising loss \eqref{eq:denoising_loss_Heisenberg} (top row) and the divergence loss \eqref{eq:loss_samples} (bottow row). For values close to the origin, the estimated scores appear visually close. Our experience is that training with the divergence loss is in practice less stable than using the denoising loss which is exemplified in the figure when the score is evaluated for higher values of $z$ where the stochastic process has lower probability density and the training therefore is performed with fewer samples.
\begin{figure}[ht]
\centering
\includegraphics[width=.45\linewidth,trim=200 25 200 75,clip]{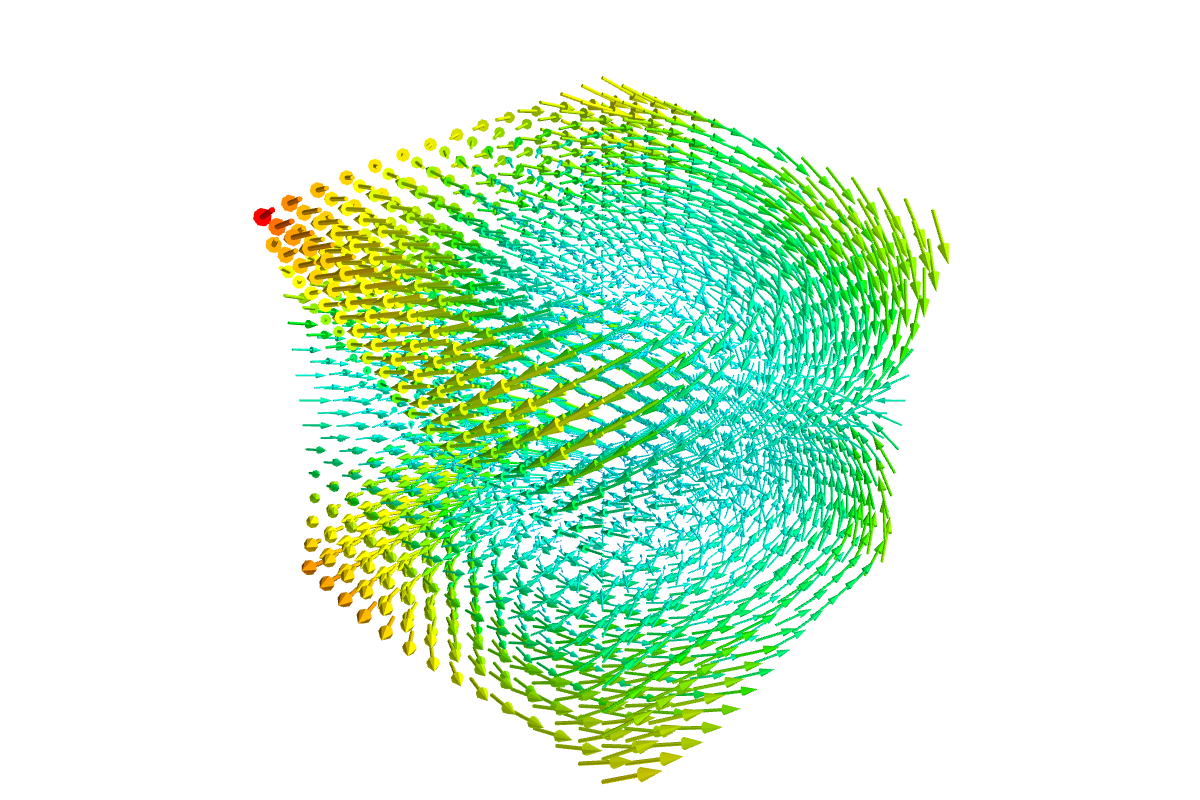}
\includegraphics[width=.45\linewidth,trim=0 0 0 0,clip]{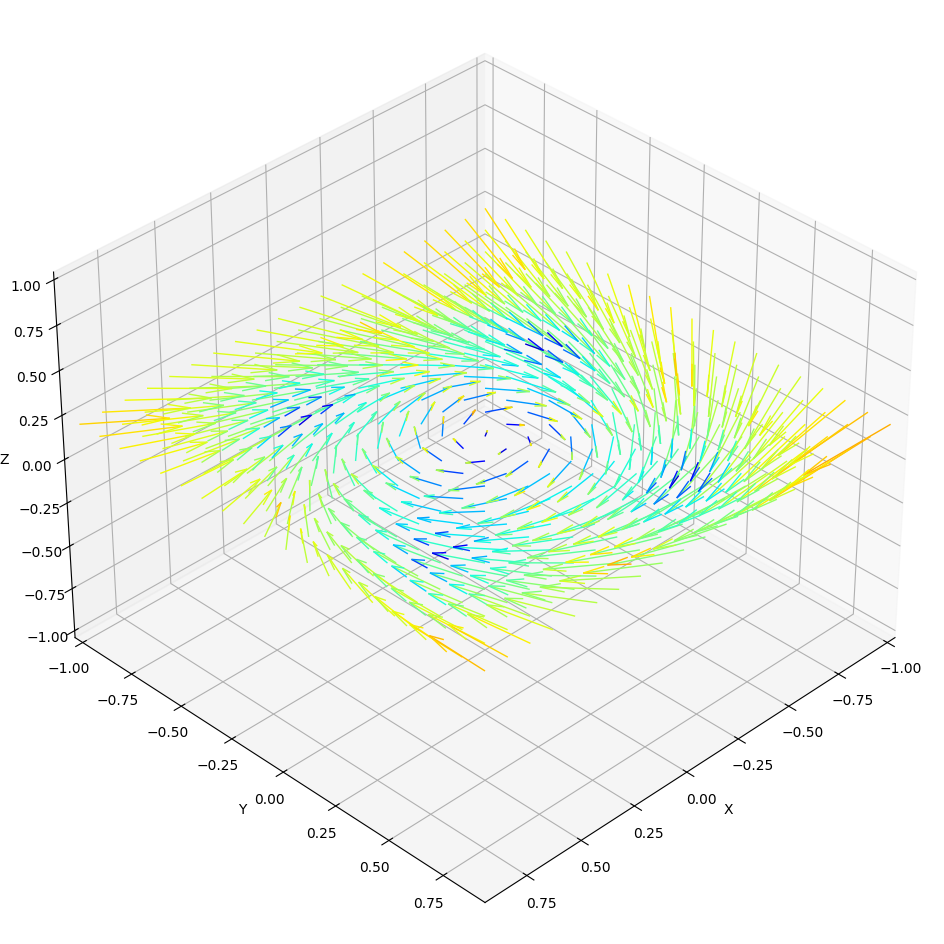}
\\
\includegraphics[width=.45\linewidth,trim=200 25 200 75,clip]{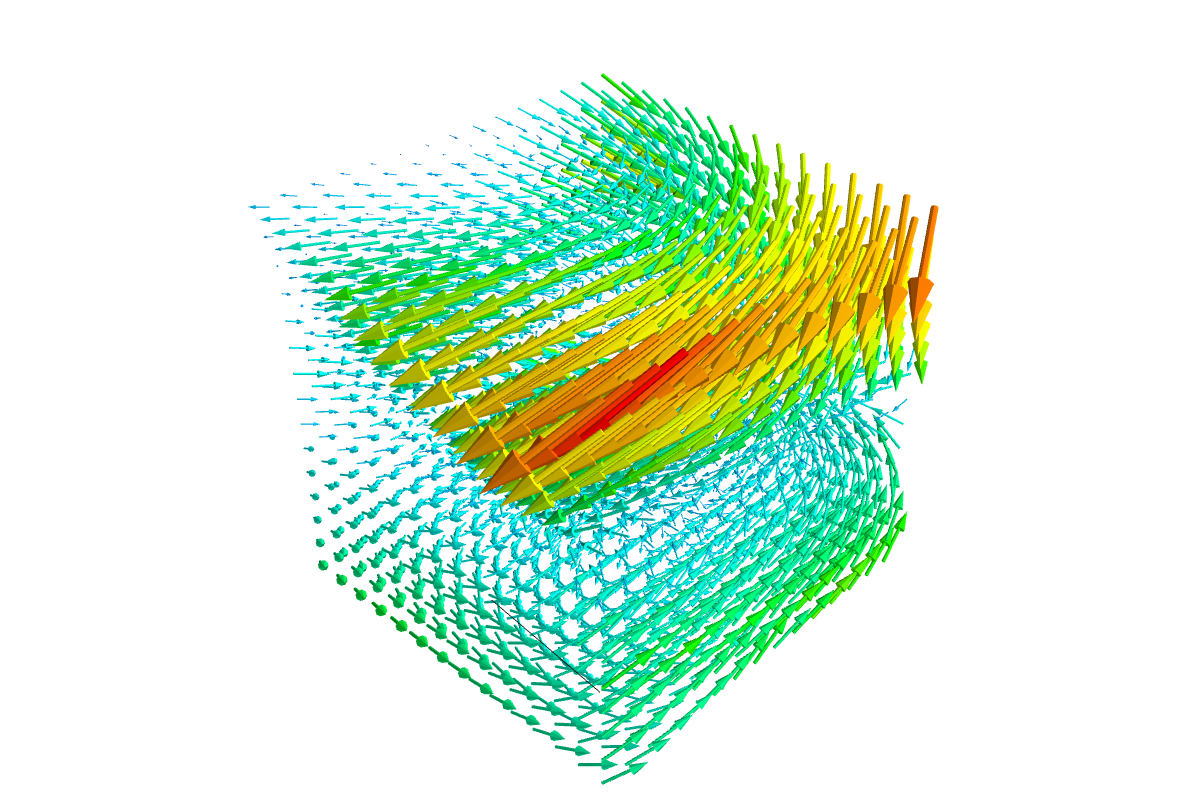}
\includegraphics[width=.45\linewidth,trim=0 0 0 0,clip]{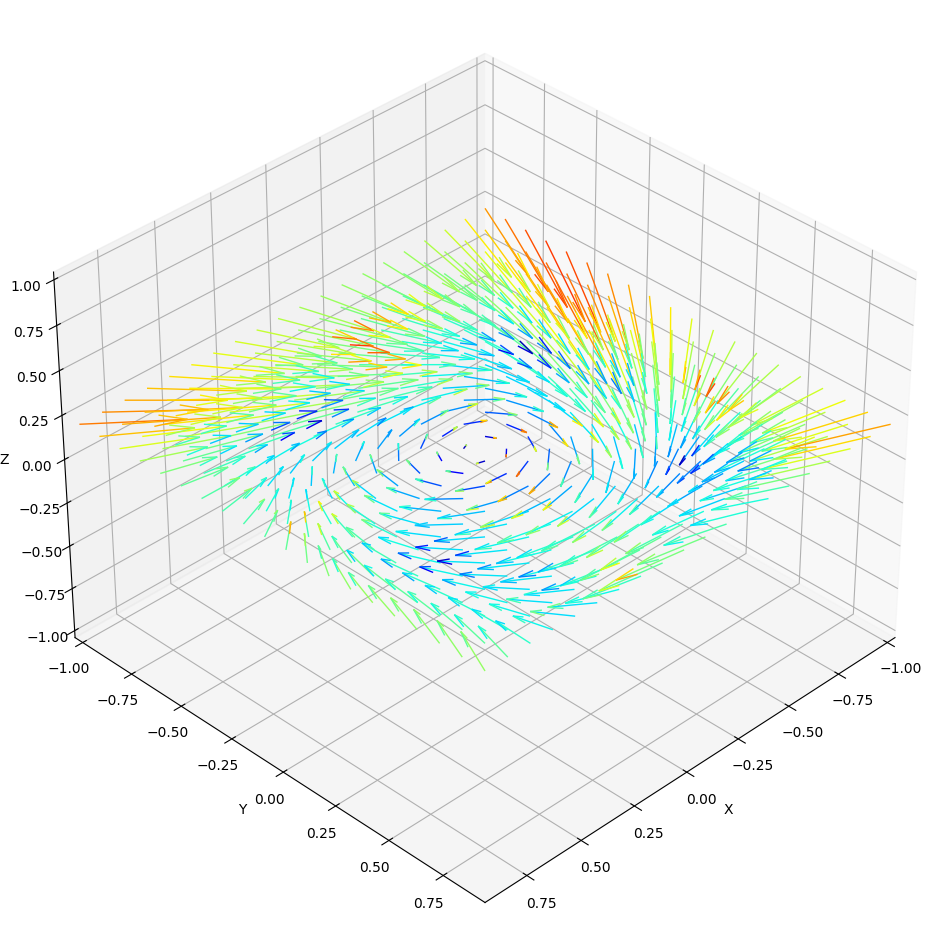}
\caption{Left: Estimated score field for $x,y,z\in[-1,1]$. Right: Estimated score for $x,y\in[-1,1]$ and $z=0.2$. Top row: Training with denoising loss \eqref{eq:denoising_loss_Heisenberg}. Bottom row: Training with divergence loss \eqref{eq:loss_samples}.
Arrow lengths scaled for visualisation, and arrows colored according to lengths.}
\label{fig:s1}
\end{figure}

\bibliographystyle{abbrv}
\bibliography{Bibliography}

\end{document}